\documentclass[12pt]{amsart}

\DeclareFontFamily{U}{matha}{\hyphenchar\font45}
\DeclareFontShape{U}{matha}{m}{n}{
  <5> <6> <7> <8> <9> <10> gen * matha
  <10.95> matha10 <12> <14.4> <17.28> <20.74> <24.88> matha12
  }{}
\DeclareSymbolFont{matha}{U}{matha}{m}{n}
\DeclareFontFamily{U}{mathx}{\hyphenchar\font45}
\DeclareFontShape{U}{mathx}{m}{n}{
  <5> <6> <7> <8> <9> <10>
  <10.95> <12> <14.4> <17.28> <20.74> <24.88>
  mathx10
  }{}
\DeclareSymbolFont{mathx}{U}{mathx}{m}{n}

\DeclareMathSymbol{\obot}         {2}{matha}{"6B}
\DeclareMathSymbol{\bigobot}       {1}{mathx}{"CB}

\usepackage[pdfauthor={Congling Qiu}, 
    pdftitle={???}%
  dvips%,colorlinks=true
  ]{hyperref}

\hypersetup{
    colorlinks=red,
    citecolor=red,
    filecolor=black,
    linkcolor=red,
    urlcolor=black
}

\usepackage[english]{babel}
\usepackage[toc,page]{appendix}

  \usepackage{multirow}

\usepackage[utf8]{inputenc}

\usepackage{xtab}
\usepackage{amsmath, amssymb%, cancel,verbatim
}
\usepackage{mathrsfs}
\usepackage[all]{xy}
\usepackage{extarrows}

\usepackage{enumerate}
\usepackage{mathtools,booktabs}
\usepackage{color}% set dimensions of columns, gap between columns, and paragraph indent 

\usepackage{epstopdf} 
\usepackage{booktabs}

\numberwithin{equation}{section}\usepackage{amsthm}
\usepackage{cleveref}

\theoremstyle{plain}
\newtheorem{proposition}{Proposition}[subsection]
\newtheorem{conj}[proposition]{Conjecture}
\newtheorem{cor}[proposition]{Corollary}
\newtheorem{lem}[proposition]{Lemma}
\newtheorem{thm}[proposition]{Theorem}
\newtheorem{prop}[proposition]{Proposition}

\theoremstyle{definition}
\newtheorem{defn}[proposition]{Definition}

\newtheorem{eg}[proposition]{Example}

\theoremstyle{remark}
\newtheorem{rmk}[proposition]{Remark}

\crefname{proposition}{proposition}{propositions}
\crefname{conj}{conjecture}{conjectures}
\crefname{cor}{corollary}{corollaries}
\crefname{lem}{lemma}{lemmas}
\crefname{thm}{theorem}{theorems}
\crefname{prop}{proposition}{propositions}
\crefname{expect}{expectation}{expectations}
\crefname{problem}{problem}{problems}

\crefname{defn}{definition}{definitions}
\crefname{construction}{construction}{constructions}
\crefname{notation}{notation}{notations}
\crefname{eg}{example}{examples}
\crefname{asmp}{assumption}{assumptions}
\crefname{hypo}{hypothesis}{hypotheses}

\crefname{rmk}{remark}{remarks}

\numberwithin{equation}{section}

  \newcommand{\BA}{{\mathbb {A}}} 
    \newcommand{\BC}{{\mathbb {C}}} 
     \newcommand{\BF}{{\mathbb {F}}}
     \newcommand{\BH}{{\mathbb {H}}}
     
     \newcommand{\BL}{{\mathbb {L}}}
     
     \newcommand{\BP}{{\mathbb {P}}}
    \newcommand{\BQ}{{\mathbb {Q}}} \newcommand{\BR}{{\mathbb {R}}}

     \newcommand{\BZ}{{\mathbb {Z}}}

     \newcommand{\cB}{{\mathcal {B}}}

    \newcommand{\cK}{{\mathcal {K}}} 
    \newcommand{\cM}{{\mathcal {M}}}

     \newcommand{\cX}{{\mathcal {X}}}

    \newcommand{\RM}{{\mathrm {M}}}

    \newcommand{\fc}{{\mathfrak{c}}}

    \newcommand{\fm}{{\mathfrak{m}}} 
     \newcommand{\fp}{{\mathfrak{p}}}

    \newcommand{\wt}{\widetilde}\newcommand{\ol}{\overline}
    \newcommand{\wh}{\widehat}

    \newcommand{\pair}[1]{\langle {#1} \rangle}

    \newcommand{\incl}{\hookrightarrow}

    \newcommand{\bsl}{\backslash}

  \newcommand{\ep}{\epsilon}
  \newcommand{\vpl}{\varprojlim}

 \newcommand{\lb}{\left(} \newcommand{\rb}{\right)}

    \newcommand{\etale}{\'{e}tale~}

    \newcommand{\Ch}{{\mathrm{Ch}}}

    \newcommand{\coker}{{\mathrm{coker}}}

    \newcommand{\End}{{\mathrm{End}}}

    \newcommand{\Gal}{{\mathrm{Gal}}} \newcommand{\GL}{{\mathrm{GL}}}

    \newcommand{\Hom}{{\mathrm{Hom}}}
    
    \newcommand{\id}{{\mathrm{id}}}

        \newcommand{\DC}{{\mathrm{DC}}}

    \newcommand{\Mor}{{\mathrm{Mor}}}

    \newcommand{\PGL}{{\mathrm{PGL}}} \newcommand{\Pic}{\mathrm{Pic}}

    \newcommand{\PSL}{{\mathrm{PSL}}}

\newcommand{\rf}{{\mathrm{f}}}
    \renewcommand{\mod}{\mathrm{mod}\ }

         \newcommand{\SL}{{\mathrm{SL}}}

    \newcommand{\Sym}{{\mathrm{Sym}}}
    \newcommand{\Stab}{{\mathrm{Stab}_e}}
     \newcommand{\wStab}{{\wt{\mathrm{Stab}}_e}}

 \newcommand{\Corr}{{\mathrm{Corr}}}

        \newcommand{\oll}{{\Omega}}
                
                     \newcommand{\aft}{{{\BA_\rf^\times}}}
\newcommand\supervisor[1]{\def\@supervisor{#1}}

\newcounter{elno}

\renewcommand{\cong}{\simeq}

 \author{Congling Qiu}
 
\begin{document}

  \title{Faber--Pandharipande    cycle, real multiplication and torsion points}
\begin{abstract}
A result of Green and Griffiths states that for the generic curve $C$ over $\BC$
 of genus $g \geq 4$ with a canonical divisor $K$, its Faber--Pandharipande 0-cycle $K\times K-(2g-2)K_\Delta$ on $C\times C$ is nontorsion in the Chow group of  rational equivalence classes.
 However, according to a conjecture of Beilinson   and Bloch,   this   Chow cycle  vanishes if the curve is defined over a number field.
   We give a proof of this prediction  for Shimura curves which have real multiplication.
Our method also  works for  some other classes curves with partial real multiplication. 
 We also draw   a connection between the Faber--Pandharipande   0-cycles
and torsion points on curves under the Abel--Jacobi map. 
 
 \end{abstract} 
\maketitle 
 \tableofcontents  
 
   \section{Introduction} 
 \subsection{Faber--Pandharipande    cycle}\label{1.1}

    Let $X$ be a geometrically   connected smooth projective
  curve of genus $g$ over a field $k$ and  $K$ a canonical divisor. Let  $K_\Delta$ be 
  the divisor $K$ on the diagonal $\Delta\subset X^2=X \times X$.
  The Faber--Pandharipande   0-cycle \cite{GG,Yin}
  $$Z_K:=K\times K-(2g-2)K_\Delta$$ 
   is of degree 0 and lies in the kernel of the Albanese map.
   Faber and Pandharipande showed that $Z_K=0$ is  rationally equivalent to 0 if $g\leq 3$.
On the other hand,
   if $k$ is  algebraically closed,  
   for the generic curve $X$ of genus $g \geq  4$, $Z_K$ is non-torsion  in the Chow group, by Green--Griffiths \cite{GG} for $k=\BC$ and by Yin \cite{Yin} in general.

  If $k$  is a number field,  according to a conjecture of Beilinson \cite[5.1, 5.6]{Bei}  and Bloch \cite[p xiv]{Blo}, 
   $Z_K$ is rationally equivalent to 0.  %However,   little is known about th conjecture, even for Faber--Pandharipande   0-cycles.       
  We give a proof of this prediction  for Shimura curves which have real multiplication.
Our method also  works for other curves which have partial real multiplication, such as  cyclic unramified coverings of genus $2$ curves of odd prime degrees, whose Prym varieties have real multiplication. 
 We also draw  a connection between the Faber--Pandharipande   0-cycles
and torsion points on curves under the Abel--Jacobi map.

 \subsection{Injectivity of    Albanese map} \label{1.2}
  For a smooth projective  variety $V$ over   $k$,   let $\Ch^i(V)$ be the Chow group  of $V$   of  codimension $i$ cycles with   $\BQ$-coefficients, modulo rational equivalence in this introduction.  
  Since   $Z_K$ 
   is of degree 0 and lies in the kernel of the Albanese map, by the    theorem of  Rojtman \cite{Roit} on torsion  0-cycles,  $Z_K$ is rationally equivalent to 0 if and  only if $Z_K$  vanishes in $\Ch^2(X^2)$. 
 Let us reformulate this vanishing as follows.

Fix     a divisor class $e\in \Ch^1(X) $ of degree 1, and
let 
\begin{equation*}
 \wh e= e\times [X] +[X]\times e,\quad \delta=   [\Delta]-\wh e\ \in\Ch^1(X^2).
 \end{equation*} 
Here and below, for an algebraic cycle $Z$, let  ``$[Z]$" denote taking  the   class   in the corresponding Chow group. Then  $[Z_K]$  
is multiple of  $(\delta\times \delta)_* \lb  [\Delta]\cdot  \wh e\rb$ 
for $e=[K]/\deg K$ (assuming $g\geq 2$). 
%vanishes in $\Ch^2(X^2)$ if and only if for $e=[K]/\deg K$ (assuming $g\geq 2$), we have
%\begin{equation} \label{eq:1.1}
%(\delta\times \delta)_* \lb  [\Delta]\cdot  \wh e\rb =0.  
%  \end{equation} 
%
Moreover,  %Take any $e\in \Ch^1(X) $ of degree 1.
 the injectivity of the   Albanese map for 0-cycles on  $X^2$  
 is equivalent to
 that  
  \begin{equation*} \label{eq:1.2}
 (\delta\times \delta)_*\Ch^2(X^2) =0 
  \end{equation*} 
   Alert reader may already notice that 
   the action of the square of $\delta $ on $\Ch^2(X^2) $ is the same with 
the square of      $[\Delta]-[X]\times e$. We use $\delta$ for later convenience.    
 
 We focus on the  subspace $$ (\delta\times \delta)_* \lb  \Ch^1(X^2)\cdot  \wh e\rb   $$ of 0-cycles, which we shall  call    generalized Faber--Pandharipande type 0-cycles with base $e$.

   \subsection{Shimura curves} \label{Shimura curves} 
   
The difficulty in the study of  Bloch--Beilinson    conjecture, which appears to be geometric,    is  to correctly use the condition  of being over a           number fields.  
  Belyi's theorem says that  any connected smooth projective curve  over   a number field can be realized as
  as the compactified quotient of the complex upper half plane $\BH$ by 
        some arithmetic subgroup of $\PSL_2(\BR)$. % Thus it would be quite challenging to deal with general arithmetic  subgroups.
We focus on some congruence  arithmetic  subgroups.

 Let  $F$ be a  totally real number field.
 An infinite (or archimedean) place of $F$ is an embedding $F\incl \BR$.
 %A finite (or non-archimedean) place $v$ of $F$ is a dense embedding of $F$ into a
%complete non-archimedean field $F_v$. Finite places are given by prime ideals of the ring of integer of $F$.
Let  $B$  be a quaternion algebra over   $F$.  
   Assume that   the base change $B_\tau$ of $B$  is
  the matrix algebra at    one infinite place $\tau $ of $F$ and division at all other infinite places of $F$.  Then 
  \begin{equation}\label{psl}
   B_\tau^+/\BR^\times\cong \PSL_2(\BR),
  \end{equation} 
  where $B_\tau^+$ consists of elements of positive norm (i.e., determinant) in the matrix algebra $B_\tau$. 
 Let $B^+ = B\cap B_\tau^+$ and  $PB^+ = B^+/F^\times$. Identify $PB^+ $ as a subgroup of $\PSL_2(\BR)$.%We define congruence   subgroups of $PB^+$ as follows.
An order   $O$ of   $B$ %and a finite place $v$ of $F$, let $O_v$ be the closure
 %of $O$ in $B_v$.
 defines an integral structure on $PB^\times = B^\times/F^\times$.
 %Let $$\Gamma_O = \{x \in PB^+: x\in O_v^\times F_v^\times\text{ for all finite places}v\}. $$
 Let $\Gamma_O\subset PB^+$ be the subgroup of integral elements. See \Cref{conncomp} for a group theoretical definition. 
 It could be larger than the image of $
 O\cap B^+$ in $PB^+$.
A  subgroup $\Gamma$ of $PB^+$ is called a
  congruence   subgroup  
if there is an order   $O$ of   $B$ such that $\Gamma_O$ is  a finite index subgroup of $\Gamma$. 
Assume that $  \Gamma$ is torsion-free and cocompact. Then
 $\Gamma\bsl \BH$  defines a  smooth projective  curve  $X_\Gamma$  over $\BC$, which is a Shimura curve over $\BC$.

 For example,   
if  $F=\BQ(\zeta_7+\zeta_7^{-1})$ and the base change  of $B$ is  the matrix algebra at
at all finite  (or archimedean)  places, then all maximal orders of $B$ are conjugate.
If $O$ is a maximal order, 
Shimura \cite{Shi} showed  
$\Gamma_O $
is the  $(2,3,7)$-triangle group in $\PSL_2(\BR)$, which is cocompact.
 % even if SL2 shimura var, can abelian group action by the quotient of GL/SL. is a unit group.  in case referee asks
So its quotients by torsion-free normal subgroups   are exactly the Hurwitz groups.  
The ones coming from prime level principal
congruence   subgroups  of $PB^+$
include the  infinite series of Hurwitz groups of the form $\PSL_2(\BF_q)$ of Macbeath \cite{Mac}.  
See \Cref{conncomp} for more details.
Hurwitz curves defined by    torsion free  congruence   subgroups  of $\Gamma_O $ are called congruence Hurwitz curves.

By the theory of Shimura varieties,  $X_\Gamma$   is canonically the geometrically connected component of a 
Shimura curve $S_\Gamma$ over $F$ (\Cref{conncomp}).

  \begin{thm}\label{thm:1}  Assume that $  \Gamma$ is torsion-free and cocompact.
 Let  $X=X_\Gamma$    of genus $g\geq 2$ and $e=[K]/\deg K$. 
  Let $ \Ch_F^1(X^2)$ be the image of  $\Ch^1(S_\Gamma^2)$ in $ \Ch^1(X^2)$ by restriction.
 Then
 $$ (\delta\times \delta)_* \lb  \Ch_F^1(X^2)\cdot  \wh e\rb  =0. $$

 \end{thm}

 As we explained in \Cref{1.2}, we apply  the theorem to $[\Delta] \in \Ch^1(X_\Gamma^2)$ so that 
 by the    theorem of  Rojtman \cite{Roit} on torsion  0-cycles, we have the following result.
  \begin{cor}\label{cor:1}Assume that $  \Gamma$ is torsion-free and cocompact.
Let  $X=X_\Gamma$    of genus $g\geq 2$.
 Then  the Faber--Pandharipande   0-cycle   $Z_K $  with integral coefficients
  is rationally equivalent to 0.

 \end{cor}

\begin{cor}\label{cor:2}
 
Let  $X$  be  a   congruence Hurwitz curve.
 Then  the Faber--Pandharipande   0-cycle   $Z_K $   with integral coefficients 
  is rationally equivalent to 0.
 
 \end{cor}

 Some remarks   are in order.
   
If $\ol\Gamma$ is a not-necessarily  torsion-free or cocompact congruence subgroup of $\PSL_2(\BR)$,  $\Gamma\bsl \BH$ has a canonical realization as a Zariski open subset of a  smooth projective 1-dimensional Deligne--Mumford stack   $ \cX_\Gamma$ over $\BC$, whose coarse moduli is still a  smooth projective curve that we denote by
 $X_\Gamma$. 
The pushfoward of a canonical divisor of  $\cX_\Gamma$ to $X_\Gamma$ 
is an orbifold  canonical divisor of $X_\Gamma$, 
which may have $\BQ$-coefficient.    
 Replacing $K$ by an orbifold  canonical divisor of $X_\Gamma$, \Cref{thm:1} still holds in this more general case (\Cref{thm:2} and \Cref{QZcusp}).
Below, in this paper, we will only consider coarse, instead of stacky, quotient. 
 
 The complementary points of $\Gamma\bsl \BH$ in $X_\Gamma$ are called the cusps,  by convention.  They exist if and only  if $X_\Gamma$ is a modular curve, i.e., when $B$ is the matrix algebra over $\BQ$. 
 In this case, \Cref{thm:1} also follows from   the   Manin--Drinfeld theorem (\Cref{MD}) on the rational equivalences among the cusps up to torsion.   
Inspired by Kato's Euler system building on explicit rational equivalences among cusps,   
we propose to study the explicit  constructions of   the rational equivalences to 0 in
\Cref{thm:1} 
for general Shimura curves.  
 \subsection{Real multiplication}\label{Real multiplication}
The proof of  \Cref {thm:1} relies on two special features of  $  X_\Gamma$.  
The first feature is that there are  many (roughly speaking)    \etale  correspondences from $X$ to $X$.
They generalize automorphisms of curves, and were explored in  our previous works \cite{QZ1,QZ2}. A new feature exploited  here is the fact that  $  X_\Gamma$ has real multiplication.

For a curve with real multiplication, its irreducible submotives in $h^1$ are essentially   self-dual of dimension 2 in the sense of Kimura and O'Sullivan \cite{Kim,OS}.
In this sense, they are just like elliptic curves. For an elliptic curve $E$ and two points $P,Q$ on $A$, it is a simple but beautiful result of Beauville and   Voisin  \cite[Lemma 2.5]{BV}
that $ (P,P)-(P,Q)-(Q,P)+(Q,Q)$ is trivial in $\Ch^2(E^2)$.  We generalize this fact in \Cref{Vanishing results} to curves with real multiplication.
 We can actually  get stronger results for curves with complex multiplication. See \Cref{prevan2}  (2) and \Cref{eg:2} (2)(3). 

Besides Shimura curves, we study two other classes of curves. First, we define $G$-multiplication  of a (branched) $G$-cover $ X\to  \BP^1$ over $\BC$, where $G$ is a finite  group,     in terms of the $G$-representation $H^1(X,\BQ)$. If it has $G$-multiplication, it has real multiplication too.
And we show the vanishing of Faber--Pandharipande 0-cycle in \Cref{Galois branched cover}.
Second, we prove the vanishing of Faber--Pandharipande 0-cycle
for cyclic unramified coverings of genus $2$ curves of odd prime degrees in  \Cref{First example}, using the fact that  their  Prym varieties have real multiplication.

 \subsection{Relation to other algebraic cycles} 
The existence of torsion points on $X$ under a canonical Abel--Jacobi embedding into its Jacobian variety implies the vanishing of the Faber--Pandharipande 0-cycle. See \Cref{Subcanonical
  curves general}.  
  We want to exclude this reason for our examples of  Faber--Pandharipande 0-cycle.
  In \Cref{MD}, we propose   suitable generalizations of the  Coleman--Kaskel--Ribet conjecture \cite{CKR,Bak} on the non-existence of torsion points  on certain modular curves to 
arbitrary Shimura curves. We give some evidence to one of our conjectures. 
In \Cref{coverhec}, we also propose a generalization of the   Poonen--Stoll theorem \cite{PS}   on the generic non-existence of torsion points  on hyperellipic curves to their odd cyclic unramified covers. We prove the ``minimal'' case of our conjecture.
The case of a $G$-cover of $   \BP^1$ with $G$-multiplication is more complicated.

 The relation between the Faber--Pandharipande 0-cycle  and the modified diagonal cycle has been discussed in \cite{QZ1}. Besides, we also noticed its deep relation  with zero cycles on Abelian varieties, which shall be discussed in a subsequent work. 
 \subsection{Structure of the paper} 
In
 \Cref  {Motivic decomposition  of a curve}, we study  0-cycles on irreducible submotives in $h^1$ 
 of a curve with real multiplication. 
 In 
 \Cref  {Faber--Pandharipande type 0-cycles},   we prove \Cref{thm:1} and analogous results for other curves mentioned above.
Finally, in 
 \Cref{Subcanonical curves}, we 
explore  the relation between the Faber--Pandharipande   0-cycles
and torsion points on curves.

    \subsection*{Acknowledgments}

The author  would like to thank  Jun~Su for introducing  
the Faber--Pandharipande cycle.   
The author  would like to thank   
 Burt~Totaro for his comments on an earlier work of the author with Wei~Zhang which inspired the main method in this paper. 
 The author  would like to thank   Qizheng~Yin  for his communication on subcanonicality. 
Finally, 
the author  would  also like to thank     Wanlin~Li, 
Ben~Moonen  and Jennifer~Paulhus
for their helps.
 
   \section{Motivic decomposition  of a curve}\label{Motivic decomposition  of a curve}
 We consider the motivic decomposition  of a curve with a large coefficient field.   Then we prove that some pieces of $(\delta\times \delta)_* \lb  [\Delta]\cdot  \wh e\rb $ vanish using results of Kimura \cite{Kim}.

  Let $\oll$ be a field   of characteristic 0.  
   Let $\Ch^i(V)$ be the Chow group  of a variety $V$   of  codimension $i$ cycles with   $\oll$-coefficients.

   \subsection{Some isomorphisms}\label{Divisorial correspondences}
% This part  is taken from our previous work \cite[Section 2]{QZ2}.

   Let $X,X'$ be  varieties over  $k$, always assumed to be connected smooth projective.  However, we do not assume geometrically connectedness.   
 Fix $e \in \Ch_0(X) ,e'\in \Ch_0(X')$    of degree 1 on each geometrically connected component.   
 \begin{defn} 
   Let $\DC^1(X\times X')$ be the subspace of  $\Ch^{1}(X \times X' )$ of $z$'s such that $z_* e=0$ and $z^*e'=0$.  
  \end{defn}
   We call  cycles  in $\DC^1(X\times X')$  divisorial correspondences with respect to  $e,e'$.

Let $A$ be an abelian variety over $k$. 
The abelian group structure of $A$ induces an abelian  group structure on $  \Mor(X,A) $, the set of morphisms from $X$ to $A$ as $k$-varieties. 
  \begin{defn}

Let $\Mor^0(X,A)_\oll$ be the subspace of  $f\in \Mor(X,A)_\oll $  such that
for some (equivalently for all) representative $\sum_i a_i p_i$ of $e_{\ol k}\in \Ch^1(X_{\ol k})$, where $a_i\in \oll$ and $p_i\in X( \ol k)$, the sum of  $ f(p_i)\otimes a_i $ over $i$'s in the group $A(\ol k)\otimes {\oll}$ is the identity.  \end{defn}

      Assume that $X$ is a curve.    Let $J=\Pic^0_{X/k}$ be the Jacobian variety of $X$. A canonical  element in $\Mor^0(X,J)$ is defined as follows.
      Assume $e = [D]/n$ where $D\in \Pic_X(k)$. Then we have a map $X\to J: x\mapsto nx-D$. Its tensor with $1/n$ is in $\Mor^0(X,J)$ and independent of $n$.
       This element plays two roles. First,  by pullback of divisors, it induces a quasi-isogeny 
\begin{equation}\label{canp}
 J^\vee\to J,
 \end{equation}     an element   in $\Hom(J^\vee,J)_\oll$. It  does not depend on the choice of $e$, since degree zero line bundles on an abelian variety is translation invariant. Indeed, it
is the negative of the inverse of the canonical polarization in $\Hom(J^\vee,J)_\oll$ (so that it has height 0), see for example \cite[17.3]{Pol}. 
Second, the Albanese property of the Jacobian variety   induces an isomorphism 
\begin{equation}\label{alb}
\Mor^0(X,A)_\oll\cong\Hom(J,A)_\oll.
\end{equation}
    % with  the Jacobian variety $J$. 

   By  \cite[Theorem 3.9]{Sch},  the following natural map  induced by pushforward  of divisors
is  an isomorphism   
\begin{equation}\label{Sch3.9}
 \DC^1(X\times X')\cong \Hom(J,\Pic^0_{X'/F})_\oll.
 \end{equation} 
 Let $ \DC^1(X\times A) $ be  defined   with $X'=A,e'$ the identity of $A$. 
By \eqref{canp}, \eqref{Sch3.9} and taking dual homomorphism, we have 
\begin{equation}\label{dualh}
 \DC^1(X\times A)\cong \Hom(A,J)_\oll.
\end{equation}

If $X'=X,e'=e$, then $\DC^1(X\times X)$  is an $\oll$-subalgebra, without sharing the same identity, of the correspondence algebra $$\Corr(X,X) := \Ch^{1}(X^2)$$
  under composition of correspondences
  $$(x,y)\mapsto x\circ y.$$ And \eqref{Sch3.9} becomes
 \begin{equation}\label{Sch3.9'}
 \DC^1(X^2)\cong \End(J)_\oll,
 \end{equation} 
  an isomorphism of $\oll$-algebras.

We generalize the definition of $ \wh e,\delta$ in \Cref{1.2} as follows. 
Let $X_{\ol k} =\coprod X_\alpha$ be the decomposition into geometrically connected components.
 Let
\begin{equation}\label{edelta}
 \wh e= (e\times [X] +[X]\times e)|_{\coprod X_\alpha^2},\quad \delta=   [\Delta]-\wh e.
 \end{equation} 
Equivalently, define $ \wh e,\delta$ as in \Cref{1.2} for each  $X_\alpha^2$ and take the union.
Then
the identity of $ \DC^1(X^2)$ is  $\delta $ and $\delta^t=\delta$ by definition.

The
 subalgebra $ \DC^1(X^2)$ of $\Corr(X,X) $ is  stable  under the involution by  taking  transpose
 of a correspondence
  $$z\mapsto z^t.$$  
 
       Moreover, we have the following well-known duality.   
\begin{lem}  \label{transdual}
 Under  the  isomorphism \eqref{canp}
 and \eqref{Sch3.9'}, taking transpose on $ \DC^1(X^2)$ corresponds to taking dual homomorphism, i.e.,  the Rosati involution %$$\phi\mapsto \phi^\dag$$  
  on $\End(J)_\oll $ associated  with  the canonical polarization.
 
 \end{lem}

\begin{lem}  \label{pushcomp2}Let $a,b, c\in \Ch^1(X^2).$
Then $(a\times b)_*c=b\circ c\circ a^t$.
\end{lem}
  \begin{proof}
This is a special case of Lieberman's lemma \cite[Lemma 3.3]{Via}.
In our case, it may be directly proved  by comparing their action on  $H^*(X)$.  
  \end{proof} 
 We have two corollaries of the above lemma.
 \begin{cor}  \label{pushcomp3}We have $(\delta \times \delta )_*  \Ch^1(X^2)= \DC^1(X^2).$
\end{cor}

  Let $H$ be a Weil cohomology theory   with coefficients in   $\oll$ (or any field extension but we do not need this generality).

  \begin{cor} \label{prevan1}
  
    Let  $\Delta\subset X^2$  be the diagonal.     
  Let $p,q\neq 0\in \Corr(X,X)$ be   projectors\footnote{Recall that 
 projectors are just
idempotents the  algebra of correspondences.}. 

(1) If  $q \circ p^t  = 0  $, then  $ (p\times q)_*   [\Delta]   = 0$.

   (2) Modulo homological equivalence, 
   $ (p^t\times p)_*   [\Delta]  $ is nonzero.
    
%(3) If  $p ^t \circ p =p  \circ p^t = 0  $ and   $p_*H^*(X)\neq 0$, then    modulo homological equivalence, 
% $ \lb (p+ p^t)\times (p+ p^t)\rb_* \lb  [\Delta] \rb  
%\neq  0$    is nonzero.
 \end{cor} 
 
% \begin{proof} 
% Only need to explain that in (2), $p+p^t$   is nonzero  modulo homological equivalence. 
% But the assumption $p ^t \circ p =p  \circ p^t = 0  $  implies
% $p_*H^*(X)\cap p^t_* H^*(X) = 0$. Thus  $(p+p^t)_*H^*(X) = p_*H^*(X)\oplus p^t_* H^*(X) \neq 0$
%  \end{proof}

In \cite[Section 2]{QZ2}, we used the duality between $\Hom(J,A)_\oll$ and $\Hom(A,J)_\oll$, as $A$ varies simple isogeny factors of $J$,   to decompose 
$\End(J)_\oll$, and thus to decompose 
the Chow motive $(X,\delta)$ via \eqref{Sch3.9'}. 
 Below, we want to extend this decomposition.  Before that, let us recall the results of Kimura we will apply after the deomposition.  
 
 \subsection{Finite-dimensional  motives}\label{Finite-dimensional  motives}
Let %$\ell$ be a prime number different from the characteristic of $k$. 
  $\cM_\oll$ be the category of Chow motives over $k$ (modulo rational equivalence) with coefficients in a field  $\oll$ of characteristic 0. 
  % the $\oll$-coefficient $\ell$-cohomology theory. 
 
 We use the the notion of ``finite-dimensionality" of motives    introduced  by Kimura \cite[Definition 3.7]{Kim} and independently by O'Sullivan \cite{OS}. 
  We will use \cite{Kim} and the results there extensively.
  Though the notion is   introduced  for  motives with coefficients in $\BQ$ in loc. cit.,  it applies to motives with coefficients in $\oll$ is the same way. See also  \cite{OS}.

 For $M\in \cM_\oll$   a finite-dimensional  motive \cite[Definition 3.7]{Kim}, let $\dim M$ be its dimension \cite[Definition 6.4]{Kim}.

 Kimura \cite[Theorem 4.2, Proposition 5.10, Proposition 6.9]{Kim} proved the following theorem.
\begin{thm} Every direct factor of the motive associated to a variety dominated by a product of curves is finite-dimensional. 
In particular this is true for an abelian variety.
\end{thm}

\begin{lem} \label{dim0}Let $M\in \cM_\oll$ be a finite-dimensional effective motive. If $H^*(M)=0$, then $  M=0$.  
\end{lem}
\begin{proof} Apply   \cite[Proposition 10.1]{Kim} to the identity morphism.
\end{proof}

 \begin{cor} \label{prevan0}   
  Let the symmetric group $S_n$ act on $X^n$ by permuting the components. 
  Let $p\in \Corr(X,X)$ be a projector. 
  Assume  $\dim p_*H^*(X)<n$.   Then for $z \in \Ch^*(X^n)$ invariant by $S_n$, we have
  $ (p^{ n})_* z  =0  $. 
 \end{cor} 

\begin{proof} 
Let $M=\Sym^n  (X,p)$.
Then by   \cite[Proposition 3.8]{Ban} and \eqref{lem:coh}, 
we have 
\begin{equation}\label{Hm}
H^*(M)= \wedge^n  \delta_{p,*}H^1(X) = 0
\end{equation} 
 Apply \Cref{dim0}.
\end{proof}

 \begin{eg}\label{eg:1}
 When $n=3$, and $z$ is the class of the diagonal curve. This is related to the modified diagonal cycles. See \cite[2.3]{QZ1}.
For example, if $X$ is has genus 1,  by \cite[(2.7)]{QZ1} and \Cref{eg:2} (1), the  proposition says that  the modified diagonal cycle is trivial in $\Ch^2(X^3)$ with any base point.

\end{eg}

\begin{lem} \label{dimH}Let $M\in \cM_\oll$ be a finite-dimensional effective motive. If $H^i(M)=0$ for $i$ even (resp. odd), then $\dim M= \dim H^*(M)$.  
\end{lem}
\begin{proof}  As $H^i(M)=0$ for $i$ even (resp. odd), 
 $H^*(\Sym^{n} M)= \wedge^n H^*(  M)$  (resp. $H^*(\wedge^{n} M)= \wedge^n H^*(  M)$), 
 see \cite[Proposition 3.8]{Ban}. We prove the even case, and the odd case is the same.
  Then $H^*(\Sym^{m+1} M)=0$ 
 for $m=\dim  H^*(M)$.
 By   \Cref{dim0}, $M=0$.
  Thus $\dim M\leq  \dim H^*(M)$. By 
\cite[Proposition 3.9]{Kim} which gives the inequality in the other direction, $\dim M= \dim H^*(M)$. 
\end{proof}

\begin{prop}\label{10.3} Let $M  \in \cM_\oll$ be an 1-dimensional effective motive. Assume $\Ch^n(M)$  is nonzero modulo homological equivalence. Then $M\cong \BL^{\otimes n}$, where $\BL$ is the Lefschetz motive.
\end{prop}
\begin{proof}
The proof is the same as \cite[Proposition 10.3]{Kim}. Note that our  assumption  
replaces the assumption on Hodge conjecture in loc. cit..
\end{proof}

  \begin{cor} \label{prevan2}   
   Let $z \in \Ch^2(X^2)$.
   Let $p \in \Corr(X,X)$ be a projector. 

   (1) Assume  $ p ^t =p $ and   $\dim p_*H^*(X)\leq 2$. If  $z$ is invariant by $S_2$, then    $ (p\times p)_*    z  =0  $

(2) Assume  $  \dim  p _*H^*(X)=\dim  p ^t_*H^*(X) \leq 1$.  Then     $ (p^t\times p)_*    z  =0  $. 
Moreover, if $z$ is invariant by $S_2$, then $ (p\times p)_*    z  =0  $.
 \end{cor} 
 
  \begin{proof}  
   (1)
 Let $M=\Sym^2  (X,p)$. Similar to \eqref{Hm}, $\dim  H^*(M)\leq 1$. 
 If $\dim  H^*(M)=0$, then  the proposition follows from 
\Cref {dim0}. 
If  $H^*(M)$ is of dimension 1. 
Then by \Cref{dimH}, we have $\dim M=1$.   
Now (1) follows from  \Cref{prevan1} (2) and \Cref{10.3}.

(2) The proof of the first part is similar to (1) with $M=  (X,p^t) \times (X,p)$ and using \Cref{prevan1} (2).
For the second part, by (1) applied to $p+p^t$,
$ \lb(p+p^t)\times (p+p^t)\rb_*    z  =0  $. Then the the second part follows from the first part.
%(2) The proof is similar to (1). Note that $p+ p^t$ is a projector. 
%Let $M=\Sym^2  (X,p+p^t)$. The  rest of the proof is the same as (1), except now that we use \Cref{prevan1} (3).
    \end{proof} 
 
  \begin{rmk}\label{eg:2}
 (1) Let $X $ be an elliptic curve over $k=\BC$,  let   $e $  be the class of a point $P$ and let $z = [(Q,Q)]$ for a point $Q$. Then (1) of the  corollary, applied to $\delta$ says $$ (\delta\times \delta)_*   z =  [(P,P)-(P,Q)-(Q,P)+(Q,Q)] =0$$  in $ \Ch^2(X^2)$. 
  This is \cite[Lemma 2.5]{BV}.

  (2) On the other hand, for a very general point $P\in X^2$, from a classical result of Mumford \cite[p 203, Corollar]{Mun}, we can deduce that 
  $ (\delta\times \delta)_*  [P] \neq 0$. 
 If $X$ has complex multiplication and $\Omega=\BC$, then $\delta =p+p^t$ with $p,p^t$
 satisfies the condition in \Cref{prevan2} (2). 
And  we   deduce that 
  $ (p\times p)_*  [P] \neq 0$ for a very general point $P\in X^2$.

   (3) The first part of (2) of the  corollary still holds if we replace 
   $p^t$ by $p'^t$ for $p'$ such that there is a nontrivial morphism between $(X,p)$ and $(X,p')$.  For example, this is the case if they come from the same CM elliptic curve. 
 However, the second part fails.  Otherwise, similar to (1), one can deduce that 
 $   [(P,P)-(P,Q)-(Q,P)+(Q,Q)] =0$  in $ \Ch^2(X^2)$ for any $X$ whose Jacobian is isogenous to a power  of a  CM elliptic curve.  This contradicts \cite[Theorem 1]{GG}.

    \end{rmk}

    \subsection{Idempotents in endomorphism algebras}  Now we can start to construct the 
 projectors (i.e., idempotents) we will use.
  We start with idempotents in endomorphism algebras of $J$.

 For  an abelian variety   $A$ over $k$, since  $\End(A)_\BQ$ is a semisimple algebra over $\BQ$,
 $\End(A)_\oll$ is a semisimple algebra over  $\oll$. See
 \cite[(5.16)]{Lam}. 
 Let  $\iota$ be an irreducible central idempotent in  $\End(A)_\oll$, and denote the corresponding simple factor as
 \begin{equation*}\label{SSS}
 S_{A,\iota} :=\iota \circ \End(A)_\oll=  \End(A)_\oll\circ\iota=\iota \circ   \End(A)_\oll\circ\iota.
  \end{equation*} 
As $\iota$ lies in the center of $\End(A)_\oll$, one can easily check that its image  $\iota^\dag$
under a  Rosati involution %$$\phi\mapsto \phi^\dag$$  
  on $\End(A)_\oll $ associated  with a polarization does not depend on the choice of the polarization.

 Let   $A$ be a  simple isogeny factor of $J=\Pic^0_{X/k}$ of multiplicity $n$ (or any  abelian variety $J$ with a polarization   in the isogeny category of abelian varieties). 
 For a central idempotent $\iota $ of $\End(A)_\oll$, the $\oll$-module $$\pi_{A,\iota}:=\iota\circ \Hom(J,A)_\oll,\quad
 \text{resp.} \quad \wt\pi_{A,\iota}:=\Hom(A,J)_\oll\circ \iota$$ is a free left  resp. right $S_{A,\iota}$-module of rank $n$, and they are dual to each other.  
 Explicitly, fix a decomposition 
 $$J=A^n\times J'$$ in   the (semisimple) isogeny category abelian of varieties
 such that $\Hom(J',A)_\BQ=0$.
 It induces $S_{A,\iota}$-basis   $\cB_{A,\iota}$  of $\pi_{A,\iota}$, i.e., 
the $i$-the element of  $\cB_{A,\iota}$ is the projection of $J$ to the $i$-th $A$ post-composed with $\iota$, and we have an
  induced dual basis $\wt\cB_{A,\iota}$, i.e., the $i$-the element of   $\wt\cB_{A,\iota}$ is the embedding of $A$  onto the $i$-th $A$ pre-composed with $\iota$.  
The image of the embedding 
  \begin{equation} \label{pMd}
    \begin{split}
  \pi_{A,\iota}\otimes_{S_{A,\iota}} \wt \pi_{A,\iota}&\incl \End(J)_\oll, \\
  \phi\otimes \psi&\mapsto  \psi\circ \phi
 \end{split}
  \end{equation} 
  is identified with a simple factor
  \begin{equation} \label{Md}
    \RM_{n}(S_{A,\iota})\subset 
 \End(J)_\oll  \cong \RM_{n}(\End(A)_\oll)\oplus  \End(J')_\oll.
   \end{equation}

  \begin{lem}  \label{lem:res}

  (1)    The image of   $\RM_{n}(S_{A,\iota})$ under  the Rosati involution %$$\phi\mapsto \phi^\dag$$  
  on $\End(J)_\oll $ is $\RM_{n}(S_{A,\iota^\dag})$.

  (2) Assume that  $\oll$ is algebraically closed and $\End(A)_\BQ$ is a field.
  Further assume 
   $\iota^\dag = \iota$\footnote{Surely, by Albert's classification, 
$\End(A)_\BQ$  is a totally real field.}. Then  the Rosati involution on $S_{A,\iota}$ is trivial. And
 there exists an invertible symmetric matrix  $g \in  \RM_{n}(S_{A,\iota})$ such that  on  $ \RM_{n}(S_{A,\iota})$,
 the Rosati involution  is given $x\mapsto g^{-1}x^Tg$, where $x^T$ is the matrix transpose of $x$.
  \end{lem}
 \begin{proof} 
 (1)    Check by definition.
 
 (2)  Let $S=S_{A,\iota}\cong \oll$ by our assumptions.  The Rosati involution is an  involution of $\oll$-algebra, and thus is  trivial on $S_{A,\iota}$.
Now (2) follow from the following variant of the Skolem--Noether theorem (applied to the matrix algebra): an involution of the   matrix algebra  $\RM_{n}(S)$  for any field $S$
 has the follwing form
$x\mapsto g^{-1}\ol x^Tg$, where $x\mapsto \ol x$ is an involution of $S$.
 \end{proof}

  \subsection{Idempotents in correspondence algebras I}\label{Field extensionI}

 Let   $A$ be a  simple isogeny factor of $J=\Pic^0_{X/k}$ of multiplicity $n$ as above.
Let $ \iota_J\in \End(J)_\oll$ be the   idempotent corresponding to the  simple factor
  $ \RM_{n}(S_{A,\iota})\subset 
 \End(J)_\oll  $. See \eqref{Md}.
 Let $$ \delta_{A,\iota}  \in\DC^1(X^2)\cong \End(J)_\oll$$  correspond to $ \iota_J$, where the last isomorphism is
  \eqref{Sch3.9}. 
  Then  \begin{equation}\label{deltai}
  \delta_{A,\iota}\circ \delta_{A,\iota}=\delta_{A,\iota},
   \end{equation}  
   and  % by \eqref {SSS} (with $$, we have 
      \begin{align} \label{deltaii}\delta _{A,\iota}\circ  \DC^1(X^2)\circ \delta _{A,\iota}=\delta _{A,\iota}\circ  \DC^1(X^2).
    \end{align}
  %  If central idempotents $\iota\neq \iota'$, then 
   % \begin{equation}\label{deltai'}
 % \delta_{A,\iota}\circ \delta_{A,\iota'}=0.
   %\end{equation}  
  By \Cref{lem:res} (1), we have
  \begin{equation}\label{deltat}
  \delta_{A,\iota} ^t=  \delta_{A,\iota^\dag}. 
 \end{equation}  
 
 Since the identity of $ \DC^1(X^2)$ is  $\delta $, we have \begin{equation}\label{dCd=dDd}
 \begin{split}
 (\delta_{A_1,\iota_1}\times \delta_{ A_2,\iota_2})_* \Ch^1(X^2)&= (\delta_{A_1,\iota_1}\times \delta_{ A_2,\iota_2})_*(\delta \times \delta )_*  \Ch^1(X^2)\\
 &=(\delta_{A_1,\iota_1}\times \delta_{ A_2,\iota_2})_*\DC^1(X^2) \\
 &=\delta_{A_2,\iota_2}\circ  \DC^1(X^2)\circ\delta_{ A_1,\iota_1^\dag},
\end{split}
\end{equation}
where for the last ``=", we applied \Cref{pushcomp2} and \eqref{deltat}.

 We write $(A_1,\iota_1)\sim ( A_2,\iota_2)$ if there is a quasi-isogeny $f$  from $A_1$ from $A_2$ that sends the central idempotent $\iota_1$  to the central idempotent  $ \iota_2$, i.e., let $f^{-1}$ be the quasi-inverse of $f$, then $\iota_2=f\circ \iota_1\circ f^{-1}$. This defines an equivalence relation among such pairs.
Then $\delta_{A_1,\iota_1}= \delta_{ A_2,\iota_2}$ if and only if $(A_1,\iota_1)\sim ( A_2,\iota_2)$. 
 And for $(A_1,\iota_1)\not\sim ( A_2,\iota_2)$, we have 
   \begin{equation}\label{5terms1} 
    \delta_{A_1,\iota_1}\circ  \DC^1(X^2)\circ\delta_{ A_2,\iota_2}=0.
     \end{equation}   
 Moreover, the sum of $\delta_{A,\iota}$'s up to the equivalence "$ \sim $" is $\delta$, i.e.,
  \begin{equation}\label{eqsum1} 
    \sum_{\{(A,\iota)\}/\sim} \delta_{A,\iota} =\delta.
     \end{equation}  
  
    \begin{lem} \label{algeprop} 
   
 (1)  Assume that $(A_1,\iota_1^\dag)\not\sim( A_2,\iota_2)$, then 
  $$(\delta_{A_1,\iota_1}\times \delta_{ A_2,\iota_2})_* \Ch^1(X^2)=0.$$
   
   (2)  For $a_1,...,a_n$ and $b_1,...,b_n$ in  $\DC^1(X^2)$ such that 
  $\sum_{i=1}^n  b_i\circ a_i^t =\delta _{A,\iota}$, we have 
\begin{align*}(\delta _{A,\iota^\dag}\times \delta _{A,\iota})_* \Ch^1(X^2)
  =(\delta _{A,\iota^\dag}\times \delta _{A,\iota})_* \DC^1(X^2)
  =\lb  \DC^1(X^2) \times \delta_{A,\iota} \rb_*   \lb \sum_{i=1}^n( a_i\times b_i)_*[\Delta]\rb.    \end{align*}

      \end{lem}
       \begin{proof} 
 (1) follows from  \eqref{dCd=dDd} and  \eqref{5terms1}.
 By (1),  \eqref{dCd=dDd}  and \eqref{deltaii}, 
  the first equation in (2)  holds, and further equals 
$\ \delta _{A,\iota} \circ  \DC^1(X^2)$.  
For the right hand side, 
by    \Cref{pushcomp2}, we have  $$ \sum_{i=1}^n( a_i\times b_i)_*[\Delta]= \delta _{A,\iota} .$$ 
By     \Cref{pushcomp2} again (and that $\DC^1(X^2)$ is   stable under   transposition), the right hand side is    $ \delta _{A,\iota} \circ    \delta _{A,\iota}  \circ \DC^1(X^2).$ 
 Now (3) follows from \eqref{deltai}. 
  \end{proof}
  
    \subsection{Idempotents in correspondence algebras II}\label{Field extensionII}

We further orthogonally  decompose $\delta_{A,\iota}$  using the  isomorphisms \eqref{alb}  and \eqref{dualh}.  

 Let $\cB_{A,\iota}$ be an $S_{A,\iota}$-basis of $\pi_{A,\iota}$ 
 and  $\wt\cB_{A,\iota}$ the dual basis of $\wt\pi_{A,\iota}$ as in the last subsection.
     For $\phi\in \cB_{A,\iota}$, let  $\wt \phi \in  \wt\pi_{A,\iota}$ be  the dual vector in the  dual basis.
  Then $ \wt \phi\circ \phi \in \End(J)_\oll$   is an idempotent. Explicitly, if  $\phi$ is the $i$-th element of $\cB_{A,\iota}$, then $ \wt \phi\circ \phi$ is the diagonal matrix with 1 on the $i$-th entry and $0$ elsewhere  of \eqref {Md}.
  Let $\delta_{\phi}\in \DC^1(X^2) $
correspond to   $ \wt \phi\circ \phi \in \End(J)_\oll$
 under the  isomorphism \eqref{Sch3.9'}.
   (An explicit construction of $\delta_{\phi}$ can be found in \cite[Section 2.5]{QZ2}.)  Then  
    we  have some properties of $\delta_{\phi}$ analogous to \eqref{deltai}, \eqref{5terms1}  and \eqref{eqsum1}
   of $\delta_{A,\iota}$.
    First,
    \begin{equation}\label{ddd}
    \delta_{\phi}\circ \delta_{\phi}=\delta_{\phi}.
     \end{equation}  
      And if $\phi_1\neq \phi_2\in \cB_{A,\iota}$ or $\phi_i \in \cB_{A_i,\iota_i}$ with 
      $(A_1,\iota_1)\not\sim ( A_2,\iota_2)$,
     then 
         \begin{equation}\label{dd0}
          \delta_{\phi_1}\circ\delta_{\phi_2}=0.
          \end{equation} 
 Moreover, 
           \begin{equation}\label{eqsum2} 
    \sum_{\phi\in \cB_{A,\iota}}  \delta_\phi=\delta_{A,\iota}.
     \end{equation}  
Besides,  by definition and Poincar\'e duality, we have  
  \begin{equation}\label{lem:coh}
  \delta_{\phi,*}H^*(X)= \delta_{\phi,*} H^1(X) \cong \iota _*H^1(A).
 \end{equation}

 We do not have analogs of \eqref{deltaii} and \eqref{dCd=dDd}, and not have \eqref{deltat} either. But we will study transpose of $\delta_\phi$ in \Cref{lem:norm}.
    
  \subsection{Vanishing results} \label{Vanishing results} 
In  this subsection, assume that $\oll$ is algebraically closed and $\End(A)_\BQ$ is a field.  
  \begin{prop} \label{prop:van0}   
  Let the symmetric group $S_n$ act on $X^n$ by permuting the components. 
  Assume  $2\dim A<   n[\End(A)_\BQ:\BQ]$.   Then for $z \in \Ch^*(X^n)$ invariant by $S_n$, 
  $$ (\delta_\phi^{ n})_* z  =0  $$  for $\phi \in \cB_{A,\iota}$. 
 \end{prop} 
 
 \begin{proof}
 
Since $2\dim A< n[\End(A)_\BQ:\BQ]$,  $\oll$ is algebraically closed and $\End(A)_\BQ$ is a field, we have $\dim   \delta_{\phi,*} H^*(X)= \dim \iota _*H^{1}(A)<n$ by \eqref{lem:coh}.
  Apply \Cref{prevan0}.
\end{proof}

  When $n=2$, the condition $\dim A<   [\End(A)_\BQ:\BQ]$ of the  proposition  is too restrictive.  
We will prove some results in this direction when $\dim A=    [\End(A)_\BQ:\BQ]$, i.e., $A$ is of $\GL_2$-type. See \Cref{prop:van1}. We need some preparations.

  \begin{lem}\label{lem:norm} Assume $\iota^\dag = \iota$. Then  there exists a choice of the basis $\cB_{A,\iota}$ such that 
  $$\delta_{\phi}=\delta_{\phi}^t.$$
 \end{lem}
 \begin{proof}   Under the assumption, $S_{A,\iota}\cong \oll$.  
A choice of the basis
 $\cB_{A,\iota}$    and 
 the induced dual basis  $\wt\cB_{A,\iota}$
give $\pi_{A,\iota}\cong \oll^n$,  the space of column vectors  
    of length $n$, and the dual $\wt \pi_{A,\iota}\cong\oll_n,$  the space   of row vectors 
    of length $n$.  The proposition is a change of basis  as follows.
     
   Let  $g$  be as in   \Cref{lem:res} (2)  which is invertible symmetric.
 Then by   \Cref{lem:res} (2),  
  we only need to show that there exists a basis $\{v_1,...,v_n\}$ of $\oll^n$ with dual basis $\{\wt v_1,...,\wt v_n\}$  of $\oll_n$   such that the matrix product $  v_i   \wt v_i \in  \RM_{n}(\oll)$ satisfies 
   \begin{equation} \label{hv}
  %v_i   \wt v_i   v_i   \wt v_i = v_i   \wt v_i,\quad
     g^{-1} \lb v_i   \wt v_i\rb^T g= v_i   \wt v_i.
   \end{equation}  
   Such  a basis $\{v_1,...,v_n\}$  is given as follows. Let $e_i\in \oll^n$ (resp. $\wt e_i\in \oll_n$) have $1$ on the $i$-th entry and $0$ elsewhere. 
          Write $g=o^{-1} do$ in $\GL_n(\oll)$ where $o$ is orthonormal (i.e., $oo^T=1$)  and $d$ is diagonal.  
          Let $v_i=o^{-1} e_i$ and $\wt v_i=\wt e_i o$.  Then \eqref{hv} can be checked directly. 
 \end{proof}
In other words,  the basis is chosen so that $ \wt \phi\circ \phi$, which is 
  the diagonal matrix with 1 on the $i$-th entry and $0$ elsewhere  of \eqref {Md}   if  $\phi$ is the $i$-th element of $\cB_{A,\iota}$,
 is fixed by the Rosati involution on $J$.

  \begin{prop} \label{prop:van1}   
   Let $z \in \Ch^2(X^2)$.
  
  (1) Assume $\iota^\dag = \iota$ and  choose $ \cB_{A,\iota}$ as in  \Cref{lem:norm}. If $\dim A\leq [\End(A)_\BQ:\BQ]$ and  $z$ is  invariant by $S_2$. 
  then 
  $$ (\delta_\phi\times \delta_\phi)_* z  =0  $$  for $\phi \in \cB_{A,\iota}$. 
    
  (2)  If   $2\dim A\leq [\End(A)_\BQ:\BQ]$,  then  
  $$ \lb \delta_\phi ^t \times   \delta_\phi\rb_* z =0  $$
 for $\phi \in \cB_{A,\iota}$.

 \end{prop}

  \begin{proof}
    The proof is similar to the one of \Cref{prop:van0} with $\dim   \delta_{\phi,*} H^*(X)=2$ resp. 1, for (1) resp. (2). This time, we apply   \Cref{prevan2}.   
   \end{proof} 
 
          \section{Faber--Pandharipande type 0-cycles}   \label{Faber--Pandharipande type 0-cycles}    
We give  applications of results   in the last section to curves with many \etale correspondences. We present  explicit examples.
     We continue to use the notations in the last section.
We do not assume that $\oll$ is algebraically closed or $\End(A)_\BQ$ at this moment. 
  \subsection{Correspondences}\label{Generalized divisorial correspondences}
   The standard direct sum decomposition:  
   \begin{equation}\label{refm1}
    \Corr(X^2)=\Ch^1(X^2)=\DC^1(X^2)\oplus   (\Ch^1(X)  \times [X]) \oplus    ([X]  \times \Ch^1(X))
    \end{equation}
gives three lemmas. 
\begin{lem}\label{lem:DCDC1}
Let $z \in \Corr(X^2)$. Then its projection to $\DC^1(X^2)$   is
$  \delta\circ z\circ \delta.$

\end{lem}
\begin{lem}\label{lem:DCDC2}
Let $z \in \Corr(X^2)$. The
composition $\Corr(X^2) \to \DC^1(X^2) \cong \End(J)_\oll$,
where the last isomorphism is \eqref{Sch3.9'}, sends $z$ to the endomorphism of $J$   induced by pushforward  of divisors.

\end{lem}
\begin{lem}\label{lem:DCDC3}
Let $c,c' \in \DC^1(X^2)$. The 
\begin{align*}
\lb  c\times c'\rb_* \lb \Ch^1(X^2) \cdot  \wh e\rb =  \lb  c\times c'\rb_* \lb \DC^1(X^2) \cdot  \wh e\rb.
\end{align*}

\end{lem}

%
%\begin{defn}
%  Let $\DC(X^2)$ be the subspace of  $\Ch^{1}(X^2)$ of $z$'s such that $z_* e\in \oll e$ and $z^*e\in \oll e$.   
%  \end{defn}
%
%We have the following direct sum decomposition:  
%   \begin{equation}\label{refm2}
%    \DC(X^2)=\DC^1(X^2)\oplus \Omega \cdot (e  \times [X]) \oplus  \Omega \cdot ([X]  \times e).
%    \end{equation}
%Then clearly, $\delta$ is a central idempotent in $ \DC(X^2)$.
%\begin{lem}\label{lem:DCDC4}
%Let $z \in \DC(X^2)$. Then its projection to $\DC^1(X^2)$,  is
%$$ \delta\circ z=z\circ \delta= \delta\circ z\circ \delta.$$
%In particular, if $z=z^t$ 
%composition $\DC(X^2) \to \DC^1(X^2) \cong \End(J)_\oll$,
%where the last isomorphism is \eqref{Sch3.9'}, sends $z$ to the endomorphism of $J$   induced by pushforward  of divisors.
%
%
%\end{lem}

    \subsection{Stabilizing correspondences}\label{Good correspondences}
We say that a morphism $$\pi=(\pi_1,\pi_2):C\to X\times X$$ from a curve $C$  is   \textit{$e$-stabilizing}  of degree $d$  
if $\pi_1^*e=\pi_2^*e$ and $\pi_i^*e$ is of degree $d>0$.  
If $e$ is clear from the context, we simply call it \textit{$e$-stabilizing}.
This implies 
$\pi_1^*[X]=\pi_2^*[X]=d[C]$.
Assume that $c$ (resp. $c'$) is  the direct image   of  $[C]$
by a $e$-stabilizing  correspondence $\pi=(\pi_1,\pi_2)$ 
 of degree $d$ (resp. $\pi'=(\pi'_1,\pi'_2):$ of degree $d'$) with respect to  $e$.
Then for $z\in \Ch^1(X^2)$, 
we have  \begin{align*}
dd'\lb \lb c\times c'\rb_* z\rb\cdot (e\times [X]  )=&dd' \lb \lb \pi_2\times\pi_2' \rb _*
\lb\pi_1\times \pi_1'\rb^* z \rb\cdot (e\times [X]  )\quad\text{by definition}
\\
=&  \lb \pi_2\times\pi_2' \rb _* \lb
\lb\pi_1\times \pi_1'\rb^*  z \cdot  \lb \pi_2\times\pi_2' \rb ^* (e\times [X])\rb \quad\text{by projection formula}
\\
=&  \lb \pi_2\times\pi_2' \rb _* \lb
\lb\pi_1\times \pi_1'\rb^*  z \cdot  \lb \pi_1\times \pi_1' \rb ^* (e\times [X])\rb\quad\text{by assumption on }\pi,\pi'
\\
=&dd'  \lb \pi_2\times\pi_2' \rb _*
\lb\pi_1\times \pi_1'\rb^* \lb z \cdot (e\times [X]  ) \rb  \quad\text{by projeIction formula}
\\
=&dd'   \lb c\times c'\rb_* \lb z \cdot (e\times [X]  ) \rb\quad\text{by definition}.
\end{align*}
Similarly,  \begin{align*}
dd'\lb \lb c\times c'\rb_* z\rb\cdot ([X] \times  e)=  dd'   \lb c\times c'\rb_* \lb z \cdot ( [X]\times e  ) \rb.
\end{align*}
So
\begin{align}\label{goodcycle0}
 \lb \lb c\times c'\rb_* z\rb\cdot \wh e=  \lb  c\times c'\rb_* \lb z \cdot  \wh e\rb.
\end{align}

\begin{defn}\label{defngoodcycle}
 (1) Let $ \Stab(X^2)\subset \Corr(X^2)$ be the $\oll$-\textit{subalgebra} generated  by  cycle classes of the  images of  
 $e$-stabilizing  morphisms. 
 
 (2) Let $ \wStab(X^2) \subset \Corr(X^2)$ be the $\oll$-\textit{subalgebra} generated  by $ \Stab(X^2)$ and $\delta$. 
\end{defn}
 % Then we have  $\oll$-subalgebras  $$\Stab(X^2)\subset \wStab(X^2) \subset \DC(X^2)\subset \Corr(X,X) ,$$ sharing  the same identity $[\Delta]$.

\begin{rmk}
In some cases, $\Stab(X^2) = \wStab(X^2)$, i.e., $\delta\in \Stab(X^2)$. For example, when a finite group $G$ acts on $X$  such that $e$ is $G$-invariant.
Then we have a morphism $\Omega[G]\to \Stab(X^2)$ sending $g\in G$ to the graph $\Gamma_g$ of $g$. If  
$H^1(X)$ has no nonzero $G$-invariants (i.e., $X/G\cong \BP^1$), then 
$e$ is unique (and is a multiple of the canonical class). Moreover,
by 
\eqref{refm1} and \Cref{lem:DCDC2}, 
 the image of the idempotent $\frac{1}{|G|}\sum_{g\in G} g$   
associated to the trivial representation
of $\Omega[G]$ in $\Stab(X^2)$ is  $\wh e/2$.  Thus $\delta\in \Stab(X^2)$. 
We will encounter this case  soon.
 For another example, see  \eqref{manyH}.
However, our discussion does not depend on this.
\end{rmk}

By \eqref{goodcycle0},  for $c,c'\in  \Stab(X^2)$ and $z\in \Ch^1(X^2)$,  we have
 \begin{align*} 
 \lb \lb c\times c'\rb_* z\rb\cdot \wh e=  \lb  c\times c'\rb_* \lb z \cdot  \wh e\rb.
\end{align*} 
\begin{defn}
  Let $\DC(X^2)$ be the subspace of  $\Ch^{1}(X^2)$ of $z$'s such that $z_* e\in \oll e$ and $z^*e\in \oll e$.   
  \end{defn}
It is direct to check that for $z\in \DC(X^2)$, we have
\begin{align*}
 \lb \lb \delta \times [\Delta]\rb_* z\rb\cdot \wh e=  \lb  \delta \times [\Delta] \rb_* \lb z \cdot  \wh e\rb,\quad
 \lb \lb [\Delta]\times\delta \rb_* z\rb\cdot \wh e=  \lb  [\Delta]\times\delta \rb_* \lb z \cdot  \wh e\rb.
\end{align*} 
Then  for $c,c'\in  \wStab(X^2)$ and $z\in \DC(X^2)$, we have
 \begin{align}\label{goodcycle}
 \lb \lb c\times c'\rb_* z\rb\cdot \wh e=  \lb  c\times c'\rb_* \lb z \cdot  \wh e\rb,
\end{align} 
where $\wh e$ is defined in \eqref{edelta}.

\subsection{First application}

For the first application, we only consider the divisor $  [\Delta]      $, but not the whole $\Ch^1(X^2)$,     and thus insensitive to field extensions. 
Assume that $k=\ol k$ for now.

Let     $X'$ also  be a  curve  and $ e'\in \Ch_0(X)$    of degree 1. 
Let $\Delta'$ be the diagonal of $X'^2$.  Define $\wh e'\in \Ch^1(X'^2)$ and
   $\delta'\in \DC(X'^2)$ as in \eqref{edelta} accordingly.
Assume that there is a morphism $f:X\to X'$ of degree $d$ such that $e'=f_* e$. 
 Let  $ \Gamma_f\in \Ch^1(X\times X')$ be the class of the graph of $f$.
It is direct to check that   $e'=f_* e$ implies
  $$ \Gamma_f\circ \delta = \delta' \circ \Gamma_f.$$ 
So 
$$\wt p_f :=   \Gamma_f^t\circ \Gamma_f\circ \delta =   \Gamma_f^t\circ \delta' \circ \Gamma_f ,$$ and thus
satisfies  
   \begin{equation}\label{pffpf}
  \wt p_f^t=\wt p_f.
       \end{equation} 
  Then $$ \wt p_f =    \delta \circ \Gamma_f^t\circ \Gamma_f =  \frac{1}{\sqrt d} \delta \circ \Gamma_f^t\circ \Gamma_f\circ \delta \in \DC^1(X^2),$$
where the last equation follows from that  $\wt p_f=\wt p_f\circ \delta$ by definition.

 Assume $\sqrt d\in \oll$.
By the projection formula, 
$\Gamma_f\circ \Gamma_f^t = d[\Delta'] \in \Corr(X',X')$  and thus 
$$p_f = \frac{1}{\sqrt d}\wt p_f$$ is a projector.
%$$p_f\circ p_f = p_f$$
Let $$q_f = \delta-p_f \in \DC^1(X^2).$$
It is  a projector, and
 \begin{equation}\label{qffqf}
  q_f^t=q_f.
       \end{equation}  
  \begin{lem} \label{1applem1}  
    Assume 
    $ (\delta'\times \delta')_* \lb [\Delta'] \cdot  \wh e'\rb  =0 $ and 
$ (q_f\times q_f)_* \lb [\Delta] \cdot  \wh e\rb  =0. $
If  $ p_f\in \wStab(X^2)$ (equivalently, $q_f = \delta-p_f\in \wStab(X^2)$), then
 $$ (\delta\times \delta)_* \lb [\Delta] \cdot  \wh e\rb  =0. $$

   \end{lem}

   \begin{proof} 
By definition, $ (\delta'\times \delta')_* \lb [\Delta'] \cdot  \wh e'\rb  =0 $ implies 
$ (p_f\times p_f)_* \lb [\Delta] \cdot  \wh e\rb  =0. $
So
 we  only need to show that $ (p_f\times q_f)_* \lb [\Delta] \cdot  \wh e\rb  =0 $. But this follows from \Cref{prevan1} and \eqref{goodcycle}.
    \end{proof}
     
Another description of $p_f,q_f$ is as follows.
Let $J'=\Pic_{X'/k}^0$ and $\pi_f:J\to J', i_f:J'\incl J$ be the homomorphism  induced by the pushforward by $\Gamma_f,\Gamma_f^t$. 
Then $\pi_f\circ i_f = d\cdot \id_{J'}$  
and gives a decomposition 
$$J = J'\oplus J_f$$
with $J_f=\ker \pi_f = \coker i_f$
in   the  (semisimple) isogeny category abelian of varieties.
The projections to $J'$ and $J_f$  corresponds to $p_f,q_f$ under \eqref{Sch3.9'}.
 The Rosati involution  on $J$ (associated  with  the canonical polarization) indcues a 
  Rosati involution  on $J_f$.

    \begin{thm}\label{1appthm}  Let
  $E = E_1\oplus E_2 \subset \End(J_f)_\BQ$ be a direct sum of commutative semisimple subalgebras.
  Assume 
   \begin{itemize}
   \item[(Base)]   
    $ (\delta'\times \delta')_* \lb [\Delta'] \cdot  \wh e'\rb  =0 $,  
    \item [(RM)]   $H^1(J_f)\otimes_E E_1$ is a free  $E_1\otimes\oll$-module of rank at most 2 and the restriction of  the Rosati involution on $E_1$ is trivial,  
  \item [(CM)]   $H^1(J_f)\otimes_E E_2$ is a free  $E_2\otimes\oll$-module of rank at most 1, and

  \item[(Stab)]  
the preimage of $E$ in $\DC^1(X^2)
$ under
  the  isomorphism \eqref{Sch3.9'} is contained in $ \wStab(X^2)$.
  \end{itemize}

Then
 $$ (\delta\times \delta)_* \lb [\Delta] \cdot  \wh e\rb  =0. $$

   \end{thm}
\begin{rmk}\label{rmk:RMCM}
Assume $k$ has characteristic 0. By Albert's classification, we have the following.

(1)  Under the first  part of   assumption (RM),  the second part is equivalent to that 
$E$ is a direct sum of totally real number fields.

(2) Under assumption (CM)  holds, $E_2$ is a direct sum of CM fields.
Then ssumption (RM) holds replacing $E_2$ by its the direct sum $E_2'$ maximal totally real subfields.

Thus  assumption (RM) and assumption (CM) can be replaced by  a single assumption (RM) for $E_1\oplus E_2'$. In other words, we essentially only use real multiplication, even when complex multiplication is present. 
  
\end{rmk}
   \begin{proof} 
    Since 
    $q_f$ 
    is
    the image of $1\in E$ in $\DC^1(X^2)
$ under
  the  isomorphism \eqref{Sch3.9'}, by \Cref{1applem1}, only need $ (q_f\times q_f)_* \lb [\Delta] \cdot  \wh e\rb  =0 $. 
After base change, we may assume $\oll$ is algebraically closed. 
Let $p,q\in \DC^1(X^2)$ be projectors correspond to two irreducible idempotents of $E\otimes_\BQ\oll$ under    the  isomorphism \eqref{Sch3.9'}. 
We only need prove this claim: $ (p\times q)_* \lb [\Delta] \cdot  \wh e\rb  =0. $
 We have several cases. If $q = p^t\in E_1\otimes_\BQ\oll$ (so that $q=p$ by \Cref{transdual}),
by   \Cref{prevan2} (1), $ (p\times q)_* \lb [\Delta] \cdot  \wh e\rb  =0. $
If $q = p^t\in E_2\otimes_\BQ\oll$, by   \Cref{prevan2} (2),  the claim holds.
 If $q\neq p^t$,  by \Cref{prevan1} and \eqref{goodcycle},  the claim holds. 
    \end{proof}

    \subsection{ 
 Galois branched cover of $\BP^1$}\label{Galois branched cover}
The simplest use case of \Cref{1appthm}  is probably when $f:X\to X'=\BP^1$ is a (branched) $G$-cover  and 
$e$ is $G$-invariant (i.e.,  a multiple of the canonical class), where $G$ is a finite  group.
Then $p_f=0$ and $J_f=J$ so that the assumption (Base) holds trivially. 
If   $\BQ[G]\to \End(J)_\BQ$ is surjective, by
  \Cref{lem:DCDC1} and \Cref{lem:DCDC2},  the assumption (Stab) holds. If moreover the assumptions (RM) and (CM) hold, then 
  $ (\delta\times \delta)_* \lb [\Delta] \cdot  \wh e\rb  =0. $ 
  The most famous example is probably the class of Fermat curves \cite{KR,Lim}. However, there is a much more elementary reason for   this vanishing. See \Cref{Subcanonical
  curves general}.

Examples as extreme as Fermat curves are rare. 
We want to weaken the  condition that $\BQ[G]\to \End(J)_\BQ$ is surjective, and want to give a criterion for assumptions (RM) and (CM) to hold. 
Assume $k=\BC$ and 
consider the Betti chomology $  H^1(X,\BQ)$ as a $G$-representation. 
This is known as the Hurwitz representation and $J$ decomposes as a product of abelian subvarieties  according to the decomposition of $H^1(X,\BQ)$ as a $G$-representation, as follows \cite{Paul}.

Fix a Wedderburn decomposition
\begin{equation}
  \label{Wedd}
  \BQ[G] = \bigoplus_i \RM_{n_i}(D_i),
\end{equation} 
where $D_i $ is a simple division  algebra over $\BQ$.
Then $V_i: = D_i^{n_i}$ is realized an irreducible rational $G$-representation.
If $D_i$ is a totally real number field, then
the multiplicity of $V_i$ in $H^1(X,\BQ)$ is even by complex conjugation. 
\begin{defn}
  The   (branched) $G$-cover $ X\to  \BP^1$ is said to have $G$-multiplication if 
  \begin{itemize}
    \item for $D_i$   a totally real number field,  
the multiplicity of $V_i$ in $H^1(X,\BQ)$ is 2,
 \item for $D_i$   a CM number field,  
the multiplicity of $V_i$ in $H^1(X,\BQ)$ is 1,
\item otherwise, 
the multiplicity of $V_i$ in $H^1(X,\BQ)$ is 0.
  \end{itemize}
\end{defn}

 Let $e:\BQ[G]\to \End(J_i)$ be the natural map induced  by pushforward of divisors. 
Let $\pi_{i,j}$ denote the idempotent of $\BQ[G]$ which, under \eqref{Wedd}, is zero everywhere except the $i$-th component where it is the matrix with 1 in the $(j,j)$ position and zeros elsewhere.   
   Then $e(\pi_{i,j}) J$'s are isomophic for the same $i$, which we denote by $J_i$.
And $$\dim J_i =  \frac{1}{2} \dim e(\pi_{i,j}) H^1(X,\BQ).$$
 Moreover, $e\lb D_i^{\oplus n_i}\rb\subset \End(J_i)$ if $J_i\neq 0$. 
 Here $D_i$ is the one in   $\BQ[G]$ corresponing to $\pi_{i,j}$ under \eqref{Wedd},i.e.,
 $\pi_{i,j}\circ \BQ[G]\circ \pi_{i,j}$.
So we have the following.
\begin{prop}
  If   the  $G$-cover $ X\to  \BP^1$  has $G$-multiplication, then $J$ has real multiplication.
\end{prop}
  
By \Cref{lem:DCDC1} and \Cref{lem:DCDC2}, the assumption (Stab) in 
 \Cref{1appthm} also holds with respect  to the real multiplication obtained in the last proposition. By  \Cref{1appthm} and the remark following it, we have the following.
\begin{thm}\label{0egthm} Let $e$ be the canonical class   of $X$ divided by its degree.
If   the  $G$-cover $ X\to  \BP^1$  has real multiplication, then
  $ (\delta\times \delta)_* \lb [\Delta] \cdot  \wh e\rb  =0. $
\end{thm}
\begin{eg}
  Assume  the  $G$-cover $ X\to  \BP^1$  has $G$-multiplication. 
If $D_i=\BQ$ or an imaginary quadratic field, then $J_i$ is  isogenous to a multiple of an elliptic curve. 
So the simplest case of  the  $G$-cover $ X\to  \BP^1$  is when $\dim J_i\neq 0$ only when $D_i=\BQ$.  Then $J$ is a  isogenous to direct sum of elliptic curves (a completely decomposable Jacobian). 
In \cite[Table 2, Table 3]{Paul},  such examples that are also nonhyperelliptic   are found  in genus up to 10, except 8\footnote{The genus 8 family in loc. cit. actually does not exist, as one can deduce from the character table of the corresponding group. Or one can check LMFDB directly. This has been communicated to the author.}. In genus 8, from LMFDB, we found such a curve with automorphism group $\PGL_2(\BF_7)$.

\end{eg} 
  \subsection{Cyclic unramified coverings of  hyperelliptic curves}\label{First example}
  We follow the discussion in \cite{Ort} and \cite[Section 2]{NOPS}.   Assume that $d$ is odd and  coprime to the characteristic of $k = \ol k$. 
 Let $f:X\to X'$ be 
 a cyclic unramified covering of degree $d$ where $X'\to \BP^1$ is hyperelliptic of  genus $g' \geq2$.
Let  $\sigma:X\to X$  an   order $d$ automorphism generating the Galois group $\Gal(X/X')$.
 Then $X\to \BP^1
 $ is a Galois branched cover, and  the hyperelliptic involution of 
 $X'$ lifts to an involution $\tau\in \Gal(X/\BP^1)$
 such that $\Gal(X/\BP^1)$ is the dihedral group of order $2d$ generated by $\sigma$ and $\tau$
 with $\tau^2 = \sigma^d=1,\tau\sigma \tau = \sigma^{-1}$. (However, the current situation differs from the last section of a $G$-cover over $\BP^1$ in the map $f$.)
Since $d$ is odd, such lifts are unique up to conjugation, given by $\tau\sigma^i$ for $i=0,...,d-1$.
     Let $C_i = X/\pair{\tau\sigma^i}$ for $i=0,...,d-1$, which are all isomorphic of genus $ \frac{(d-1)(g'-1)}{2}$,
      and $J_i=\Pic_{C_i/k}^0$ is an abelian subvariety of $J_f$ (in the isogeny category, 
      where we defined $J_f$\footnote{Actually, one can define the Prym variety $J_f$ in the usual way in stead of the  isogeny category, 
      and the statements here and below still hold.}).  
      Then  
      \begin{equation}\label{J0J1}
            J_0\cap J_1 = 0,\quad J_f=J_0\times J_1. 
      \end{equation}

\begin{thm}\label{1egthm} Let $e$ be the canonical class   of $X$ divided by its degree.
If $g'=2$ and $d$ is an odd prime, then
  $ (\delta\times \delta)_* \lb [\Delta] \cdot  \wh e\rb  =0. $
\end{thm}
\begin{proof}

  Let $\zeta$ be a primitive $d$-th root of unity.
  By    \cite[Remark 2.3]{NOPS}, we have a subalgebra
  $  \BQ(\zeta+\zeta ^{-1})\subset \End(J_i)_\BQ$, which is a totally real field.
  By \cite[Proposition 4.1]{NOPS},  the restriction of  the Rosati involution on  $  \BQ(\zeta+\zeta ^{-1})$ is trivial.
   When $g'=2$ and $d$ is an odd prime, $[\BQ(\zeta+\zeta ^{-1}):\BQ]=\dim J_i = (d-1)/2$. 
   We want to apply \Cref{1appthm} with 
   $$E_1=\BQ(\zeta+\zeta ^{-1})^{\oplus 2}\subset \End(J_0)\oplus \End(J_1)$$ and $E_2=0$. 
    The assumption (BASE) follows from the hyperellipticity of $X'$ and the unramifiedness of $f$ (so that $e'$ is the class of a Weierstrass point).
      The assumption (RM) follows from the discussion above.
      The assumption (CM) is void.
It remains to show that the preimage of $ \BQ(\zeta+\zeta ^{-1})\subset \End(J_i)_\BQ$  in $\DC^1(X^2)
$ under
  the  isomorphism \eqref{Sch3.9'} is contained in $ \wStab(X^2)$, and it is enough for $i=0$.
For this, we need to recall the construction of the subalgebra 
$  \BQ(\zeta+\zeta ^{-1})\subset \End(J_0)_\BQ$ in \cite[Remark 2.3]{NOPS}.
Indeed, $\sigma^i+\sigma^{-i}$ stabilizes  $J_0$ and $\zeta^i+\zeta ^{-i}$ is its preimage in $\End(J_0)_\BQ$. 
 Then by \Cref{lem:DCDC1} and \Cref{lem:DCDC2},   the image of $\zeta^i+\zeta ^{-i}$ in $\DC^1(X^2)$  is 
$  \delta\circ z\circ \delta,$ 
where $z$ is the graph of  $ \frac{1+\tau}{2} \circ (\sigma^i+\sigma^{-i})\circ\frac{1+\tau}{2}  $.
So it is in $ \wStab(X^2)$
\end{proof}

 \subsection{Second application}
 To prove the vanishing of more general Faber-Pandharipande type 0-cycles, we need some preparations. 
 
 \begin{prop}\label{prop:good}

(1)  Assume that $(A_1,\iota_1^\dag)\not\sim ( A_2,\iota_2)$ and 
 $\delta_{A_1,\iota_1}, \delta_{ A_2,\iota_2} \in \wStab(X^2)$.
 Then 
  $$(\delta_{A_1,\iota_1}\times \delta_{ A_2,\iota_2})_*  \lb  \Ch^1(X^2)\cdot  \wh e\rb=0.$$
  
 (2)     For     $a_1,...,a_n$ and $b_1,...,b_n$ in  $ \wStab(X^2)$ such that 
  $\sum_{i=1}^n b_i\circ a_i^t =\delta _{A,\iota}$,  we have 
    \begin{align*} (\delta _{A,\iota^\dag}\times \delta _{A,\iota })_*\lb   \Ch^1(X^2) \cdot  \wh e\rb=  \lb\DC^1(X^2) \times \delta_{A,\iota}\rb_*   \lb \sum_{i=1}^n( a_i\times b_i)_*\lb [\Delta]\cdot  \wh e\rb \rb.
    \end{align*}

\end{prop}

  \begin{proof} 
  By \Cref{lem:DCDC3} and \eqref{goodcycle}, we have 
  $$(\delta_{A_1,\iota_1}\times \delta_{ A_2,\iota_2})_* \lb  \Ch^1(X^2)\cdot  \wh e\rb
  = \lb(\delta_{A_1,\iota_1}\times \delta_{ A_2,\iota_2})_*  \DC^1(X^2)\rb \cdot  \wh e.$$
  Then  
(1) follows from  \Cref{algeprop} (1). And the left hand side of the equation in (2) becomes 
$ \lb(\delta _{A,\iota^\dag}\times \delta _{A,\iota })_*  \DC^1(X^2)\rb \cdot  \wh e$.
   Now we prove (2). By    \Cref{algeprop}  (2), we have    $$ \lb(\delta _{A,\iota^\dag}\times \delta _{A,\iota })_*  \DC^1(X^2)\rb \cdot  \wh e=\lb \lb  \DC^1(X^2) \times \delta _{A,\iota}\rb_*   \lb \sum_{i=1}^n( a_i\times b_i)_*[\Delta]\rb\rb \cdot  \wh e.$$
  Applying \eqref{goodcycle} again and twice, this equals
   \begin{align*} 
    \lb\DC^1(X^2) \times \delta_{A,\iota}\rb_*   \lb \sum_{i=1}^n( a_i\times b_i)_*\lb [\Delta]\cdot  \wh e\rb \rb.
    \end{align*}  
     \end{proof}
    \begin{thm}\label{thm:goodcycle}
      Assume 
 for any  simple isogeny factor $A$ of $ J_{X}$,  $\End(A)_\BQ$  is a field such that
    \begin{itemize}
      \item[(RM)]    $\dim A\leq  [\End(A)_\BQ:\BQ]$ and
       the restriction of  the Rosati involution on $A$ is trivial,
     or 

  \item[(CM)]   
 $2\dim A\leq  [\End(A)_\BQ:\BQ]$.
  \end{itemize}
  If   $\DC^1(X^2)\subset \wStab(X^2)$,
then $$ (\delta\times \delta)_* \lb  \Ch^1(X^2)\cdot  \wh e\rb  =0. $$

   \end{thm}

  \begin{rmk}
An analog of \Cref{rmk:RMCM} on  Albert's classification holds here too.
\end{rmk}

     \begin{proof} 
     By \Cref
     {prop:good} (1),  
   if  $(A_1,\iota_1^\dag)\not\sim( A_2,\iota_2)$, then 
  $$(\delta_{A_1,\iota_1}\times \delta_{ A_2,\iota_2})_* \lb  \Ch^1(X^2)\cdot  \wh e\rb =0.$$
   By  \eqref{eqsum1} (the sum of $(A,\iota)$'s up to the equivalence "$ \sim $" is $\delta$),  we only need to show that 
    \begin{equation}\label{final1}
    (\delta_{A,\iota^\dag}\times \delta_{ A,\iota})_* \lb  \Ch^1(X^2)\cdot  \wh e\rb =0.
       \end{equation} 
   We want to apply \Cref
     {prop:good}  (2). We need to choose $a_i,b_i$ in \Cref
     {prop:good}  (2).
      By 
\eqref{eqsum2} and
\eqref {ddd}, we may choose $a_i=b_i^t$ and $\{b_1,...,b_n\}=\{\delta_\phi:\phi\in \cB_{A,\iota}\}$ for  any basis $\cB_{A,\iota}$ as in loc. cit..
Then by  \Cref
     {prop:good}  (2), only need  
     \begin{align}\label{final} (\delta_\phi^t\times \delta_\phi)_*
 \lb [\Delta]\cdot  \wh e\rb=0 \text{ for }\phi\in \cB_{A,\iota}.
 \end{align}

  Since it is enough to prove the theorem in the case $\oll$ is algebraically closed, we may assume this is the case from now on.
    If assumption (RM) holds, we can apply   \Cref{lem:norm} to choose $\cB_{A,\iota}$  such that $\delta_\phi^t= \delta_\phi$.  
  Then we can apply \Cref{prop:van1} (1) to conclude \eqref{final}.
If assumption (CM) holds, we can apply \Cref{prop:van1} (2) to conclude \eqref{final}. 
    \end{proof}

     \subsection{Shimura curves over a totally real number field}\label{pfthm1}
 We apply \Cref{thm:goodcycle} to Shimura curves.  
First, we recall basics on Shimura curves, following 
  \cite[Section 3]{YZZ}.   
  
  Let  $B$  be a quaternion algebra over a  totally real number field $F$. 
 Assume that   $B$ is  the matrix algebra at   one infinite place $\tau:F\incl \BR$ of $F$ and division at all other infinite places of $F$ as in \Cref{Shimura curves}.
 Let $B^+$ act on $\BH$ via \eqref{psl}.   
 
 Consider $G = B^\times$ as an algebraic group over $F$ whose center is $F^\times$.
  Let     $\BA_\rf$  be the ring of finite adeles of $F$.   
    For    an open compact-modulo-center subgroup $U$ of $G(\BA_\rf)$, we have the smooth   compactified   Shimura curve $ X_U$ for  $G$ of level $U$ over $F$.   The  complex uniformization  of  $X_U$ via $\tau: F\incl \BR\subset \BC$ is given by
\begin{equation}\label{compluni} X_{U,\BC} \cong B^+ \bsl \BH \times G(\BA_\rf)/U\coprod \{\text{cusps}\},\end{equation}
Here  
 the cusps exist if and only if $X_U $ is a modular curve, i.e. $F=\BQ$, and $B $ is the matrix algebra. In this case, the quotient is not proper
 and the cusps are the points at infinity for compactification.

   For $g\in G(\BA_\rf)$, we have  the right multiplication isomorphism $\rho(g) :X_{gUg^{-1}}\cong X_U$.  Indeed, under  the above  complex uniformization  of  $X_U$, $\rho(g)$ acts by right multiplication on $G(\BA_\rf)$. 
For example, if $g $  is in 
   the center ${\aft}$ of $G(\BA_\rf)$,
   then $\rho(g)$ is an isomorphism on $X_U$ and is the identity if $g\in F^\times(U\cap {\aft})$. Thus the finite abelian group 
   ${\aft}/F^\times(U\cap {\aft})$ acts on $X_U$ by $\rho$. %This is the
%canonical finite abelian group  action mentioned right above \Cref{thm:1}.

   We also have      the natural projections $$\pi_1:X_{U\cap gUg^{-1}}\to X_U,\quad \pi_2:X_{U\cap gUg^{-1}}\to X_{gUg^{-1}}.$$
 Define the Hecke correspondence $Z(g)$ to be  the algebraic cycle on $X_U^2$ that is  the direct image of the fundamental   cycle   of $X_{U\cap gUg^{-1}}$ via $(\pi_1,\rho(g)\circ\pi_2)$. 
  % sufficiently small (the last condition is to remove the effect of the center).
In particular, $Z(1)$ is the diagonal of $X_U^2$.  
   % Let $$\HC(X^2)\subset \Ch^1(X^2)$$ be spanned by $[Z(g)]$'s as $g$ varies.
     We will see that Hecke correspondence stabilizes the Hodge class that we now define (following \cite[3.1.3]{YZZ}).
For $x \in X_U(\BC)$ that is the image of a point in $\BH \times G(\BA_\rf)/U$ 
with a nontrivial stabilizer,  we call $x$ an elliptic point and let $r_x$ be  the order of the stabilizer. There are only finitely many elliptic points.
Let $\fc$ be the divisor of the cusps of $X_U$, which is 0 unless $X_U$ is a modular curve.
    Let $K_U$ be the canonical divisor of $X_U$.
Let
$$L_U = K_U + \sum_{\text{elliptic } x } \lb 1-\frac{1}{r_x}\rb x+\fc.$$ 
 Then $\deg L_U>0$
 and $L_U$
 is compatible under pullback and pushforward (up to degree) by  $\rho(g)$, as well as  the natural morphisms of Shimura curves as the level  $U$ changes.   
 Let $$e = [L_U]/\deg L_U^\circ$$ where   $L_U^\circ$ is the restriction of $L_U$ to 
 any geometrically connected component of $X_U$.
Then $e$ has degree 1 when restricted to each geometrically 
connected component of $X_U$. 
And $$[Z(g)]\in \Stab(X).$$

By \cite[Proposition 3.2.2]{QZ2} , we have 
     \begin{equation*}  \DC^1(X^2)  \subset \Omega\{[Z(g)], g\in G(\BA_\rf)\} .
     \end{equation*} 
     So  
      \begin{equation}\label{manyH} 
        \DC^1(X^2)  \subset \Stab(X) =\wStab(X).
             \end{equation} 
     To apply \Cref{thm:goodcycle}, we verify the assumptions (RM)  for $X_U$.
 
\begin{lem}[{\cite[Theorem 3.4, Theorem 3.8]{YZZ}}] \label{Areal}Let $A$  be a  simple isogeny factor of $ J_{X_U}$. Then $\End(A)_\BQ$  is
 a   field and
 $\dim A= [\End(A)_\BQ:\BQ]$.
Moreover, if $U$  containes the center ${\aft}$ of $G(\BA_\rf)$, then
$\End(A)_\BQ$ is a totally real number field.
\end{lem}
\begin{rmk} If $U$  containes the center ${\aft}$ of $G(\BA_\rf)$, 
  $X_U$ is a Shimura curve   for the reductive group $PB^\times$. 
\end{rmk}

Now by \Cref{thm:goodcycle} and the remark following it,
  we conclude that the following theorem holds.

  \begin{thm}\label{thm:2}
    Assume   $U$  containes the center ${\aft}$ of $G(\BA_\rf)$.
For  $X=X_U$ and $e$ as above, we have
 $$ (\delta\times \delta)_* \lb  \Ch^1(X^2)\cdot  \wh e\rb  =0. $$

 \end{thm}
      \begin{rmk}\label{QZcusp}
 In \cite[Remark 3.2.4]{QZ2}, we further noticed that if $X_U$ is a modular curve, then  
$\fc$ is a multiple of $e$. So if
$$\cK_U = K_U + \sum_{\text{elliptic } x } \lb 1-\frac{1}{r_x}\rb x $$ 
which is an orbifold  canonical divisor of $X_U$, then $e = [\cK_U]/\deg \cK_U^\circ,$
provided that $\deg \cK_U>0$
      \end{rmk}

     \subsection{Shimura curves over $\BC$ and Hurwitz curves}\label{conncomp}
         The image of $\BH \times \{1\} $ in $X_U$ under \eqref{compluni}  is
     isomorphic to $B^+\cap U\bsl \BH$.
     With cusps in its complex analytic closure, it becomes  a 
     geometrically connected component  of $X_U$, and denoted by  $X_{B^+\cap U}$. 
       If  $U$  containes the center ${\aft}$ of $G(\BA_\rf)$,
       then  $  F^\times\subset B^+\cap U $ and  $B^+\cap U /F^\times$     is a congruence subgroup
     of $PB^+$ 
     in the sense of \Cref{Shimura curves}, and any congruence subgroup of $PB^+$ comes from this.
     If $O$ is an order of $B$,
     let $$\wh O = \prod_{v \text{ finite}}  O_v,$$ where $O_v$ be the closure
 of $O$ in $B_v$.
Then $ \lb B^+\cap \wh O^\times \aft \rb /F^\times$ is $\Gamma_O$ defined in \Cref{Shimura curves}.

     \begin{prop}[{\cite[(3.12.1) and p 82]{Shi}}]
   Assume   $F=\BQ(\zeta_7+\zeta_7^{-1})$ and the base change  of $B$ is  the matrix algebra at
at all finite places.
Let  $O$ be a maximal order  of $B$,
unique up to  
conjugation. Then 
$\Gamma_O $
is the  $(2,3,7)$-triangle group in $\PSL_2(\BR)$.
 
     \end{prop}
In the rest of this subsection, 
fix the set up in the proposition.
We show  that $X_{B^+\cap {V\aft }}  $'s give the infinite series of Hurwitz curves of Macbeath \cite{Mac} for prime level principal congruence subgroups $V$ of $\wh O^\times $. 
In particular, their congruence subgroups   are 
Hurwitz groups. 

   Let $B^1$ be the subgroup of $B^\times$ consisting of elements  of reduced norm 1.
  Then $B_\tau^1\cong \SL_2(\BR)$.
    For a compact  subgroup $V$ of $B^1(\aft)$, where $B^1$ is considered as an algeraic group over $F$,
      $B^1\cap V $ is a congruence subgroup of $B^1$. 
We will
 need the following  fact.
 \begin{lem}\label{eq:sub}
  For any subgroup $V$ of $G(\BA_\rf)$, we have
   \begin{equation*} B^+\cap {V\aft } \subset N_{B^+}(B^+\cap V)   \subset N_{B^+}(B^1\cap V).
\end{equation*}
 \end{lem}
\begin{proof}
For the first inclusion, let  $x \in B^+\cap V$ and $g\in B^+\cap {V\aft } $. Since $x\in V$ and $g\in V\aft$, 
$g^{-1}xg \in V$. Since $x\in B^+$ and $g\in B^+$, 
$g^{-1}xg \in B^+$. So $g\in N_{B^+}(B^+\cap V) $. 
The second inclusion follows from that conjugation keeps norm.
\end{proof}

 For a  finite place $v$  of $F$, let   $\fp_v$ be the maximal ideal of the ring of integers $R_v$ of $F_v$. 
 Let 
 $$
 \wh\Gamma_O(v) =  \{x\in \wh O^\times: x \equiv 1 \ \mod \fp_v\},
 $$
 $$
  \wh\Gamma_O^{\pm }(v) =  \{x\in \wh O^\times: x \equiv \pm 1\ \mod \fp_v\}
  $$
 and 
  $$
  \wh\Gamma_O^*(v) =  \{x\in \wh O^\times: x \in (R_v/\fp_v)^\times\ \mod \fp_v\}.
  $$
Let $O^1 = B^1\cap O$. 
By \cite[Lemma 4.3]{Dza}, we have  
$N_{B^+}(B^1\cap \wh\Gamma_O(v))=O^1F^\times $.
Then by \Cref{eq:sub},  
$B^+\cap {\wh\Gamma_O(v)\aft } \subset O^1F^\times $. 
Thus 
\begin{align*}
  &B^+\cap {\wh\Gamma_O(v)\aft } 
= O^1F^\times \cap {\wh\Gamma_O(v)\aft } 
= \lb O^1 \cap {\wh\Gamma_O(v)\aft }\rb F^\times  \\
= &\lb O^1 \cap \wh\Gamma_O^*(v)\rb F^\times
= \lb O^1 \cap \wh\Gamma_O^{\pm 1}(v)\rb F^\times
= \lb O^1 \cap \wh\Gamma_O(v)\rb F^\times
.
\end{align*}
Here on the second line, the first second equation is from the integrality imposed on  $\aft$ by intersection with $O^1$,
the second follows from that  square roots of 1 in $R_v/\fp_v$ are $\pm 1$. Other equations are by elementary group theory arguments. 
 By \cite[Corollary 4.2, Lemma 4.4]{Dza}, we have the following. 
\begin{prop}\label{Hurwitz}
 (1) For an open compact-modulo-center subgroup $U$ of $G(\aft)$ containing $\aft$ and contained  in $\wh\Gamma_O(v)\aft$ for some   finite place  $v$  of $F$, 
$X_{U}$ is a Hurwitz curve.

  (2) As $v$ varies over finite places of $F$,  $X_{ B^+\cap {\wh\Gamma_O(v)\aft }}$
gives the  infinite series of Hurwitz curves of Macbeath \cite{Mac}. 
Moreover, assume $v$ has residue characteristic $p$. Then its Hurwitz group is $\PSL_2(\BF_q)$ with 
\begin{itemize}
\item  $q=p$ if $p=7$ or $p\equiv \pm1 \ \mod 7$;
\item $q=p^3$ if $p\equiv \pm 2,3 \ \mod 7$.
\end{itemize}

\end{prop}

 %%%%%%%%%%%%%%%%%%%%%%%%%%%%%
 
 %%%%%%%%%%%%%%%%%%%%%%%%%

 \section{$\BQ$-subcanonical curves} \label{Subcanonical curves}

 Assume $k=\BC$. 
 Let $X$ be a smooth projective curve of genus $g\geq 2$ over $\BC$

\subsection{$\BQ$-subcanonical points and torsions} \label{Subcanonical
  curves general}

A point $P\in X$ is   subcanonical if   $(2g-2)p$ 
is a canonical divisor. Then subcanonical points are Weierstrass points. 
Conversely, not all Weierstrass points are necessarily subcanonical, but on hyperelliptic curves, they are. 
 A more sophiscated class of examples are from the Fermat curves. See \cite[Exercise 1, E-2, E-3, E-11]{ACGH}. 
The moduli space of  curves with a subcanonical points
was studied by
 Kontsevich and  Zorich \cite{KZ}, as 
 the minimal dimensional stratum of the moduli space of Abelian differentials.

More generally, a point $P\in X$ is   $\BQ$-subcanonical if   $[p]=[K]/\deg K$ in $\Ch^1(X)$, 
where $K$ is a canonical divisor of $X$. 
The moduli space of  curves with a $\BQ$-subcanonical point
can also be studied in terms of differentials, as initiated by Farkas and Pandharipande \cite{FkP}.

We also consider $\BQ$-subcanonicality from a more number-theoretic point of view.
Let $\xi\in \Pic^1_{X/\BC}$   such that $(2g-2)\xi$ is the class of $K$.
Let $i_\xi: X\incl \Pic^0_{X/\BC}$ be the associated embedding with base $b$.
Then $P\in X$ is $\BQ$-subcanonical if and only if $i_\xi(p)$ is a torsion point of $\Pic^0_{X}$.
By the uniform Manin--Mumford \cite[Theorem 2]{Kuh} (for base divisors of degree 1, instead of base points), the   orders of  torsion points of $\Pic^0_{X/\BC}$ contained $X$ is bounded by a constant depending only on $g$.
Thus, we have the following theorem.
\begin{thm}\label{thm:subcan}
  There is a constant depending only on $g\geq 2$ 
  such that the number of $\BQ$-subcanonical points on a smooth projective curve   of genus $g $ over $\BC$ is bounded by this constant.  
\end{thm}

With this torsion point interpretation of $\BQ$-subcanonical points, we can give more examples of $\BQ$-subcanonical points.
 Coleman \cite[Section VI]{Col} defined two points $P,Q\in X$ to be ``in the same torsion packet" if $P-Q\in \Pic^0_{X/\BC}$ is torsion. And he proved that if 
 $f:X\to \BP^1$ is an abelian branched covering with $B\subset \BP^1$ the branched locus, 
 Then $f^{-1}(B)$ form a  torsion packet. See also \cite{Col1}.
 Since we can take a canonical divisor of $X$   supported on $f^{-1}(B)$, we have the following proposition.
\begin{prop}\label{Colprop}
Let $f:X\to \BP^1$ be an abelian branched covering with $B\subset \BP^1$ the branched locus. 
Then the preimage of any point in $B$ is a $\BQ$-subcanonical point of $X$.
\end{prop}
Again, Fermat curves, which we have seen to be subcanonical,  are examples to which  this proposition applies. The degree $n$ 
Fermat curve is a $\BZ/n\BZ\times \BZ/n\BZ$ cover of $\BP^1$.
Another example is Klein's quartic. See \Cref{Klein}.

Define $X$ to be $\BQ$-subcanonical if there is a $\BQ$-subcanonical point $P\in X$.
For a $\BQ$-subcanonical  curve,  the Faber--Pandharipande   0-cycle
$$Z_K:=K\times K-(2g-2)K_\Delta$$ is 0 in $\Ch^1(X^2)$. 
Then by  Green--Griffiths \cite{GG}, a generic curve of genus $g> 2$ 
 is not $\BQ$-subcanonical.

We expect  that most of our above examples of  curves are \textit{not }$\BQ$-subcanonical.  
In the remaining two subsections, we only formalize our expectations in
explicit conjectures for Shimura curves (as well as Hurwitz curves) and unramified cyclic coverings of hyperelliptic curves. We also give
give some evidence. 
The case of curves with $G$-multiplication
is more complicated due to the generality and the complicatedness of the corresponding moduli.  \Cref{Colprop} requires that a necessary condition for a 
 $G$-cover to avoid being $\BQ$-subcanonical
 is that it
  can not be realized as an abelian cover of $\BP^1$.   We ask if this is also a sufficient condition.

  \subsection{Shimura curves} \label{MD}

   We consider general congruence  subgroups  of $\PSL_2(\BR)$. 
A subgroup $\Gamma$ of $\PSL_2(\BR)$ is a congruence  subgroup if it is the image of 
a congruence  subgroup of some $B^1$ as defined in \Cref{conncomp}.
 Define an orbifold  canonical divisor on $X_\Gamma$ as in \Cref{QZcusp}, and define 
 orbifold $\BQ$-subcanonicality of $X_\Gamma$ accordingly. In particular, if    $\Gamma$  a \textit{torsion-free} congruence  subgroup, this is the same with $\BQ$-subcanonicality.

 If   $X_\Gamma$ is a modular curve, i.e., when $B$ is the matrix algebra over $\BQ$, 
 then $X_\Gamma=\Gamma\bsl \BH \coprod \{\text{cusps}\}$. 
By  the   Manin--Drinfeld theorem, the cusps on $X_\Gamma$ are rationally equivalently to each other  up to torsion, or  ``in the same torsion packet" in Coleman's terminology. 
Similar to 
\Cref{QZcusp}, one can then show the following.
\begin{lem}\label{cuspsub}
If  $X_\Gamma$ is a modular curve, then it
 is orbifold $\BQ$-subcanonical, and its orbifold $\BQ$-subcanonical points are the ones rationally equivalently to  the cusps up to torsion, or ``the cuspidal torsion packet" in Coleman's terminology. 
\end{lem}
\begin{cor}\label{cuspsubcor}
If  $X_\Gamma,X_{\Gamma'}$ are  modular curves with $\Gamma\subset \Gamma'$. Then  the natural map $X_\Gamma\to X_{\Gamma'}$ sends
  orbifold $\BQ$-subcanonical points to   orbifold $\BQ$-subcanonical points.
\end{cor}

In other cases, Shimura curves have no cusps. Then, motivated by the fact that there are only finitely many hyperelliptic Shimura curves \cite{Qiufin}, we propose the following.

\begin{conj}\label{conj1}There are only finitely many orbifold $\BQ$-subcanonical Shimura curves, up to isomorohism,  over $\BC$ that are not modular curves.
\end{conj}

 We want to propose a variant of \Cref{conj1} for modular curves.

 \begin{conj}\label{conj2} 
  
There are only finitely many modular curves, up to isomorohism,  over $\BC$ that  
have orbifold  $\BQ$-subcanonical points outside cusps.
\end{conj}
We give a piece of   evidence to \Cref{conj2}.
  The  congruence subgroup 
\begin{align*} 
 \Gamma_0(N)&= \bigg\{
 \begin{bmatrix}
 a & b \\
 c & d
 \end{bmatrix}\in\SL_2(\BZ): c\equiv 0 \mod N\bigg\}.
 \end{align*}
 in $\SL_2(\BZ)$ defines the
   classical modular curve $X_0(N)$. 
   
  Let   $P_1 $ be the set 
of primes less than $23$, and $P_2= \{23, 29, 31, 41, 47, 59, 71\}.$
  If $p\geq 23,p\not\in P_2$ is a prime number,  the
 Coleman--Kaskel--Ribet conjecture \cite{CKR} (under \Cref{cuspsub})
 claims that $X_0(p)$ has no orbifold $\BQ$-subcanonical points  except cusps.
It was proved by Baker and Tamagawa \cite{Bak}.  As a corollary  (see Proposition 4.1 in loc. cit), for such a $p$ and a   positive integer $M$, 
$X_0(pM)$ has no orbifold $\BQ$-subcanonical points  except cusps.
\begin{thm}\label{thmtorsion}
  If $X_0(N)$ has   orbifold  $\BQ$-subcanonical points  except cusps, then $N=Mq$ 
    where $M$ is a product of primes in $P_1$ and $q\in P_2$ or    $q=1$.
 
\end{thm}
\begin{proof}

By the discussion above the theorem, we may assume that $N$ is a product of primes in $P_1\coprod P_2$, 
Moreover, if $p$ is a prime in $P_2$, it was proved in loc. cit. that 
 $\BQ$-subcanonical points are fixed points of the hyperelliptic Atkin--Lehner involutions, 
 and thus are CM points whose endomorphism algebras are the same  order in 
 $\BQ(\sqrt{-N})$. The conductor of this order is called the conductor of these CM points.

 We have assumed that $N$ is a product of primes in $P_1\coprod P_2$, 
To prove the proposition, only need to prove the following 
 claim: if $N=pqM$ where $p,q\in P_2$ and $M$ is positive integer, then
$X_0(N)$ has no orbifold  $\BQ$-subcanonical points  except cusps. 
 By \Cref{cuspsubcor}, we may assume $M=1$. 
Then if $x\in X_0(pq)$ is an  orbifold  $\BQ$-subcanonical point outside the cusps, then its image in 
$X_0(p),X_0(q)$ are both CM points with CM fields
 $\BQ(\sqrt{-p}),\BQ(\sqrt{-q})$. So $x$ is a CM point with CM field  $\BQ(\sqrt{-p}),\BQ(\sqrt{-q})$, and  therefore $p=q$. 
 Recall that $X_0(p^2)$ classifies pairs $(E,C)$ where $E$ is an elliptic curve and $C$ is a  finite subgroup scheme of $E$ of order $p^2$. The standard $\phi:X_0(p^2)\to X_0(p) $
sends $(E,C)$ to $(E,C[p])$.  
Also consider $\psi:X_0(p^2)\to X_0(p) $ that sends $(E,C)$ to $(E,C/C[p])$. 
 Then the conductor of $\psi(x)$ is at least $p$ times the conductor of $\phi(x)$.
But  $\psi(x)$ is an  orbifold  $\BQ$-subcanonical point
 by \Cref{cuspsubcor} and it is outside the cusps, a contradiction.
\end{proof}
    In particular, if there are  infinitely many $N$ such that 
$X_0(N)$ has orbifold  $\BQ$-subcanonical points outside cusps, then by \Cref{cuspsubcor} for some $p\in P_1$,
$X_0(p^n)$ has orbifold  $\BQ$-subcanonical points outside cusps for every positive $n$.
Moreover, there is a sequence of orbifold  $\BQ$-subcanonical points $x_n\in X_0(p^n)$
gives a point in $\vpl_n X_0(p^n)$. For this tower of modular curves, there could be 
  more tools to study it. For example, use perfectoid spaces as in \cite{Qiu1}.

 \begin{eg}
  
We give an example of a non-hyperelliptic $X_0(N)$ with $\BQ$-subcanonical points  outside cusps. By  \cite[Theorem 2 (i)]{Ogg1} and that $3\equiv -1 \ \mod 4,\ 2\equiv -1 \mod 3$, $X_0(42)$ has no elliptic points and is non-hyperelliptic.
On  $X_0(42)$, there are 3 newforms $f_1,f_2,f_3$ of level $14,21,42$ respectively, as shown on LMFDB \cite{LMFDB}.  And each of them contribute to $2,2,1$ of the genus 5 of $X_0(42)$ by the theory of new forms. 
The unique quadratic order of conductor 2 in $\BQ(\sqrt{-3})$ has class number 1.
Then the  set of Heegner (CM) points defined by this order  in \cite[1.1 Introduction]{CST} has cardinality one. Let $\infty$ be a cusp and $x$ be this unique Heegner point. We sketch how to show that $[x]-[\infty] = 0$ in $ \Ch^1(X_0(42))$, equivalently, the $f_i$-component of $[x]-[\infty]$ is 0 for $i=1,2,3$, using   the formulation in \cite[1.3.2]{YZZ} with  invariant linear forms on the automorphic representation corresponding to $f_i$. We need to show the vanishing of these   spaces of invariant linear forms.  
At  a prime $p$ (in $2,3,7$)  dividing the level of $f_i$,
the local $\PGL_2(\BQ_p)$-representation is a twist of the Steinberg representation 
by one of the only two unramified quadratic characters $\pm 1$ of $\BQ_p^\times$. 
The twisting character is opposite to the Atkin--Lehner sign which can be read   from LMFDB. Then by \cite[Lemma 3.1 (3)]{CST}, we find the vanishing of the local (and thus global) invariant linear forms at one of $p$ dividing the level of $f_i$. 
So the $f_i$-components of $[x]-[\infty]$ is 0 for $i=1,2,3$. 
Alternatively, one may also use the general Gross--Zagier formula \cite{YZZ} in this case and read the  rank of the twisted $L$-functions of $f_i$'s by $\BQ(\sqrt{-3})$, which are 0, from LMFDB.
 \end{eg}

For   Shimura curves over $\BQ$,
 an analog of  the work of  Baker and Tamagawa \cite{Bak} on Coleman--Kaskel--Ribet conjecture, and thus
 \Cref{thmtorsion}, might be done using a similar strategy. Indeed
  the work of  Baker and Tamagawa heavily relies on the primary part of Mazur's famous work \cite{Maz} on rational torsion on $X_0(p)$. The analog of this result (the primary part) was largely proved  in \cite{Yoo}.

 In view of the relation between Hurwitz curves and Shimura curves, we also ask if the analog of \Cref{conj1} holds for Hurwitz curves.

\begin{conj}\label{conj3}There are only finitely many  $\BQ$-subcanonical Hurwitz curves.
\end{conj}
\begin{eg}\label{Klein} 
 The first Hurwitz curve, Klein's quartic $x^3y+y^3z+z^3x = 0$ with automorphism group $\PSL_2(\BF_7)$, has three avatars in our context.
 First, it is the modular curve $X(7)$ by Elkies \cite{Elk}. So its cusps are  $\BQ$-subcanonical-points. Second, from its Hurwitz representation, its quotient  by 
a subgroup of $\PSL_2(\BF_7)$ isomophic to $\BZ/7\BZ$ or $\BZ/2\BZ\times \BZ/2\BZ$ is $\BP^1$. So \Cref{Colprop} applies and gives $\BQ$-subcanonical-points. Finally, it is the Shimura curve over $F=\BQ(\zeta_7+\zeta_7^{-1})$ in \Cref{Hurwitz} (2) associated to the ramified prime over $7$. This also gives an example to \Cref{conj1}.

   \end{eg}

\subsection{Cyclic unramified coverings of  hyperelliptic curves}\label{coverhec}

The following lemma  direct from definition is useful.
\begin{lem}\label{lemunr}
  Let $Y\to X$ be an unramified covering. 
  The image of a $\BQ$-subcanonical point on $Y$ is a $\BQ$-subcanonical point on $X.$
\end{lem}
 A result of Poonen of Stoll \cite{PS} says that
 for a generic hyperelliptic curve, the only $\BQ$-subcanonical points are 
 Weierstrass points. Thus, we have the following.
\begin{thm}\label{thmunr}
  Let    $X$ be an unramified covering of a generic hyperelliptic curve.
  The only $\BQ$-subcanonical points on $X$ are points in the preimages of 
 Weierstrass points.  
 \end{thm}

Let $d$ be an odd  number. 
 Let $f:X\to X'$ be 
 a cyclic unramified covering of degree $d$ where $X' $ is hyperelliptic of  genus $g'\geq 2$.
By \cite[Proposition 2.1]{NOPS}, the moduli of such coverings is irreducible. 
 \begin{conj}\label{conjcyc}
 The generic $X$ has no $\BQ$-subcanonical points. 
 \end{conj}
 \begin{rmk}
  When $d=2$, 
  the analog of the conjecture gets complicated due to reducibility of the moduli. 
  Moreover, if $g'=2$, $X$ is also hyperelliptic. See \cite{Fak}.
 \end{rmk}
By  \Cref{lemunr}, 
 \Cref{conjcyc} holds if it holds for all $d$ odd and prime.
 We have the  ``minimal'' case of \Cref{conjcyc} in \Cref{thmcyc} below.
A key ingredient comes from Shimura curves that we now describe.

Let $B$ be a quaternion algebra over $\BQ$  that is     division only at $2,13$.
 Let  $O$ be a maximal order  of $B$,
unique up to  
conjugation.  
Let $U=\wh O^\times$. By   \cite[Theorem 7]{Ogg1}, the Shimura curve $X_{B^1\cap U}$ over $\BC$ is hyperelliptic of genus 2.
By \cite[Lemma B.0.4]{Qiufin}, we have 
the natural map $X_{B^1\cap U}\to X_{B^+\cap U}$ is an isomorohism.
 
 Let $\fm_2$ be the maximal ideal of $O_2$ (the completion of $O$ at the place $2$) and 
 $$V = (1+\fm_2)\times\prod_{v\neq 2 \text{ finite}}  O_v^\times.$$
\begin{prop}\label{propcyc}
(1) The morphism $X_{B^+\cap V}\to X_{B^+\cap U} $ is cyclic unramified of degree 3.

(2) Any point in the preimages of  Weierstrass points of $X_{B^+\cap U} $ is not $\BQ$-subcanonical.
\end{prop}
Then by  \Cref{thmunr}, we have the following ``minimal'' case of \Cref{conjcyc}.
\begin{thm}\label{thmcyc}
  For $d=3,g'=2$, \Cref{conjcyc} holds.
\end{thm}

In the rest of this subsection,  we prove \Cref{propcyc}. 
 
By \cite[(B.1), Corollary B.0.2]{Qiufin}, we have the following.
\begin{lem}\label{lemcyc}
  (1) The Shimura curves $X_U, X_V$ over $\BQ$ are geometrically connected.  In particular, $X_{B^+\cap U} = X_{U,\BC}$ and $ X_{B^+\cap V} = X_{V,\BC} $. 
  
  (2) The natural morphisms between Shimura curves $X_U\to X_{U\aft}$ and 
  $ X_V\to X_{V\aft}$  are isomorphisms. 

\end{lem}
 By   LMFDB, there is a unique  minimal newform 
of weight 2 and level 52 with trivial central character, labeled as $52.2.a.a$. 
               Denote it by $\phi$. 
Its Jacqeut Langlands correspondence is 
on  $X_{V\aft}$, but not on $X_{U\aft}$.

\begin{proof}[Proof of \Cref{propcyc} (1)] 
  
  By   \cite[Theorem 2 (i)]{Ogg1} and that $13\equiv 1 \ \mod 3,\ 13\equiv 1 \ \mod 4$, the Shimura curve $X_{B^+\cap U}$ has no elliptic points. So (for any $V\subset U$), 
$X_{B^+\cap V}\to X_{B^+\cap U} $ is unramified. 
By \cite[A.3]{Qiufin}, $[U:V]=3$ so that $X_{B^+\cap V}\to X_{B^+\cap U} $ has degree at most 3. 
  By \Cref{lemcyc} and the discussion following it, $X_{B^+\cap V}\to X_{B^+\cap U} $ is nontrivial and the proof is finished,
\end{proof}

To prove (2), we will identify the Weierstrass points as CM points on  $X_{  U} $. Then their preimages in $X_V$ are also CM points. And  we then apply the general Gross--Zagier formula \cite{YZZ} to $X_V$. 
We need some preparations for the proof. 
Let $K=\BQ(\sqrt{-26})$. The class group of $K$ is $\BZ/6\BZ$, as one can find on LMFDB. 
Then $K$ has a unique quadratic unramified extension, which is 
$L = \BQ(\sqrt{-2}, \sqrt{13})$. Let $\chi$ be the unique quadratic   class group character of $K$, corresponding to $L/K$. Then 
the base change $L$-functions of $\phi$ satisifes
\begin{equation}
  \label{Ldecom}
  L(\phi_K,s)L(\phi_K\otimes \chi,s) = L(\phi_L,s).
\end{equation}

\begin{lem}\label{lemcyc1}
   Choose an embedding  $K\incl B$ of $\BQ$-algebras. 
Let $\pi$ be the automorphic representation of $B^\times$ corresponding to the Jacquet--Langlands correspondence of the automorphic representation of $\GL_{2,\BQ}$ generated by $\phi$.
Identify $\chi$ as an idele class character. 
Then their local components at $2,13$ satisfy $$\Hom(\pi_{2}|_{K_{2}^\times},1)\neq  0,\ \Hom (\pi_{13}|_{K_{13}^\times},1)= 0$$
and 
$$\Hom (\pi_{2}|_{K_{2}^\times} \otimes\chi_{2},1)\neq  0,\  \Hom (\pi_{13}|_{K_{13}^\times}\otimes \chi_{13},1)\neq 0.$$

\end{lem}
 \begin{proof}
 
  Let $D_6$ be the dihedral group of order 6, i.e., the symmetric group of order 6. 
Then at $p=2$, $B_p^\times/\BQ_p^\times(1+\fm_p)\cong D_6$. Since $K_p$ is ramified over $\BQ_p$, the image of $K_p^\times$ is the unique subgroup of order 2 in $D_p$ up to conjugation. See, for example, \cite[A.3, Lemma A.3.1]{Qiufin}. 
The   unique 2-dimensional representation of $D_6$ is isomorphic to $\pi_p$. See for example \cite[Lemma A.1.4]{Qiufin}.  From its character table, easy to see that the restriction of $\pi_p$ to this unique subgroup of order 2 in $D_6$ is the direct sum of its two characters. Thus we have the formulas at $p=2$.
  
At $p=13$,$B_p^\times/\BQ_p^\times O_p^\times\cong \BZ/2\BZ$. 
Since $K_p$ is ramified over $\BQ_p$, the image of $K_p^\times$ is  $\BZ/2\BZ$ generated by the uniformizer of $K_p$. Since $L_p/K_p$ is inert, $\chi_p$ is the unramified character $-1$ of  $K_p^\times$, i.e.,  the character $-1$ of  $\BZ/2\BZ$.
Since the Atkin--Lehner sign of $\phi$ is $+1$ from LMFDB,  
  $\pi_{p}$, as   a character  of  $\BZ/2\BZ$,  is  the opposite $-1$.  
  Thus we have the formulas at $p=13$.

 \end{proof}

 By \Cref{lemcyc1} (1) and the root number part of the Tunnell--Saito Theorem \cite{Sai,Tun}, the local root number
 of $L(\phi_K,s)$ (resp. $L(\phi_K\otimes \chi,s)$) at 2, 13 are $-1,1$ (resp., $-1,-1$).
Since both $L(\phi_K,s)$ and $L(\phi_K\otimes \chi,s)$ have local root number $-1$ at infinity, 
their global root numbers are $1,-1$ respectively.  
In particular, $L'(\phi_K,1)=0$ and $L(\phi_K\otimes \chi,1)=0$. From $L(\phi_K\otimes \chi,1)=0$, by \eqref{Ldecom}, we have  $L(\phi_L,1)=0$.
Moreover, using Magma \cite{Magma}(with its feature  ``Leading:=true" for computing derivative), we compute
$$L(\phi_K,1) = 1.85863978993643183457996916142..., 
$$
and 
$$
L'(\phi_L,1)=5.67155941014956154756688048510....
$$
Then we conclude that 
\begin{equation}\label{Lneq0}
  L(\phi_K\otimes \chi,1)=\frac{L'(\phi_L,1)}{L(\phi_K,1)}\neq 0.
\end{equation}

 \begin{proof}[Proof of \Cref{propcyc} (1)] 
  %Points on $X_U$ has moduli interpretation as Abelian surfaces with endmorphism by $O$.
From the proof of  \cite[Theorem 7]{Ogg1}, we found that Weierstrass points  on  $X_{B^+\cap U} = X_{U,\BC}$ are given by the CM points whose endomorphism algebras are the  ring of integers  $O_K=\BZ[\sqrt{-26}]$
   of $K$. A point $x\in X_{V,\BC}$ in the preimage of a Weierstrass point is still a CM point. 
   And the endomorphism algebra of $x$ is an order $O'$ in $K$ such that $O'_p =O_p$ for 
   $p\neq 2,13$. 
   We claim that its $\chi$-twisted sum over the Galois orbit of $x$, as defined in \cite[1.3.1]{YZZ}, is nonzero in $\Ch^1(X_U)$. Then $x$ is not $\BQ$-subcanonical. 
   But this claim follows by applying the general Gross--Zagier formula \cite[Theorem 1.2]{YZZ}. 
   Besides \eqref{Lneq0}, we need the local conditions in loc. cit..  At $p=2$ and $13$, since $\pi_p^{V_p}=\pi_p$, they   hold by \Cref{lemcyc1}. They also hold at $p\neq 2,13$ by \cite[PROPOSITION 6.4]{GroLoc}.
\end{proof}

 \begin{rmk}
  Same method can potentially  be used to study $\BQ$-subcanonicality of unramified covering of hyperellipic curves by other groups, as long as there is such a cover of Shimura curves.  
 \end{rmk}

\end{document}